\documentclass[a4paper,10pt,reqno]{amsart}
\usepackage[UKenglish]{babel}
\usepackage{amsmath, wasysym}
\usepackage{amssymb}
\usepackage{verbatim}
\usepackage{stackrel}
\usepackage[arrow, matrix, curve]{xy} 

\usepackage{orcidlink}

\usepackage{caption} 
\usepackage{tabularray}

\usepackage{comment,cite,etoolbox,url}	
\usepackage{hyperref}
\hypersetup{
    colorlinks=true,
    linkcolor=blue,
    citecolor=red,
    filecolor=magenta,      
    urlcolor=cyan,
    pdftitle={Duality},
    linktocpage=true,
}
\usepackage[hyphenbreaks]{breakurl}

\usepackage[utf8]{inputenc}
\usepackage{bbm}

\usepackage[shortlabels]{enumitem}
\usepackage{marginnote}
\usepackage{color}

\usepackage{stmaryrd}
\usepackage{tikz}
\usepackage{tikz-cd}

\usetikzlibrary{arrows}
    \usetikzlibrary{graphs,decorations.pathmorphing,decorations.markings}
    \usetikzlibrary{calc}
    
    \pgfmathsetmacro\weight{1/2}%
    \pgfmathsetmacro\third{1/3}%
    \pgfmathsetmacro\twothirds{2/3}%
    
    \tikzset{degil/.style={%
              decoration={markings,
              mark= at position 0.5 with{%
                    \node[transform shape] (tempnode) {$/$};
                }%
            },%
            postaction={decorate}
    }
    }

\newcommand{\bbB}{\mathbb{B}}
\newcommand{\bbC}{\mathbb{C}}

\newcommand{\bbN}{\mathbb{N}}

\newcommand{\calD}{\mathcal{D}}

\newcommand{\calL}{\mathcal{L}}
\newcommand{\calM}{\mathcal{M}}

\newcommand{\calX}{\mathcal{X}}

\newcommand{\dotbox}{\tikz[baseline=-0.5ex]\draw[black,very thick, fill=black] (0,0) rectangle (0.3em,0.3em) node[black,pos=.5,circle,fill=white,inner sep=0.3pt]{\tiny$\cdot$};}
\newcommand{\argument}{\mathord{\,\cdot\,}} 
\newcommand{\dx}{\;\mathrm{d}} 
\newcommand{\norm}[1]{\left\lVert #1 \right\rVert} 
\newcommand{\modulus}[1]{\left\lvert #1 \right\rvert} 
\newcommand{\duality}[2]{\left\langle#1\, ,\, #2\right\rangle} 
\DeclareMathOperator{\dom}{dom} 
\DeclareMathOperator{\Ima}{Rg} 
\newcommand\restrict[1]{\raisebox{-.5ex}{$|$}_{#1}} 

\renewcommand{\sun}{\odot} 
\newcommand{\moon}{\leftmoon} 
\newcommand{\Reg}{\operatorname{Reg}} 
\DeclareMathOperator{\SV}{SV} 
\DeclareMathOperator{\Var}{var} 

\newcommand{\resSet}{\rho}
\newcommand{\Res}{R} 

\definecolor{TUBlue}{cmyk}{1,0.70,0.10,0.50} 

\theoremstyle{plain}

\theoremstyle{definition}
\newtheorem{definition}{Definition}[section]
\newtheorem{remark}[definition]{Remark}
\newtheorem{remarks}[definition]{Remarks}
\newtheorem*{remark*}{Remark}
\newtheorem*{remarks*}{Remarks}

\newtheorem{proposition}[definition]{Proposition}

\newtheorem{theorem}[definition]{Theorem}

\numberwithin{equation}{section} 

\title{Admissible operators for sun-dual semigroups}

\author[S.~Arora]{Sahiba Arora\, \orcidlink{0000-0003-1973-8358}}
\address{Sahiba Arora, Department of Applied Mathematics, University of Twente, 217, 7500 AE, Enschede, The Netherlands.
Current address: Leibniz Universität Hannover, Institut für Analysis, Welfengarten 1, 30167 Hannover, Germany}
\email{sahiba.arora@math.uni-hannover.de}

\author[F.L.~Schwenninger]{Felix L.~Schwenninger\,\orcidlink{0000-0002-2030-6504}}
\address{Felix L.~Schwenninger, Department of Applied Mathematics, University of Twente, 217, 7500 AE, Enschede, The Netherlands}
\email{f.l.schwenninger@utwente.nl}

\date{\today}

\begin{document}

\begin{abstract}
    We extend classical duality results by Weiss on admissible operators to settings where the dual semigroup lacks strong continuity. 
    This is possible using the sun-dual framework, which is not immediate from the duality of the input and output maps. This extension enables the testing of admissibility for a broader range of examples, in particular for state space of continuous functions or $L^1$.
\end{abstract}

\maketitle

\section{Introduction}

Our starting point is linear time-invariant systems of the form
\begin{equation*}
    \label{eq:lti-system}
    \Sigma(A,B,C)\quad\left\{
        \begin{aligned}
            \dot{x}(t) &= Ax(t)+Bu(t),\quad &t\ge 0\\
            y(t)       &= Cx(t),       \quad &t\ge 0\\
            x(0)       &= x_0;
        \end{aligned}
    \right.
\end{equation*}
where $x(t)$ denotes the state of the system at time $t$, $u(t)$ denotes the input, and $y(t)$ denotes the output. The state, input, and output spaces are denoted by $X, U$, and $Y$ respectively and are assumed to be Banach spaces. Moreover, $A$ is assumed to generate a $C_0$-semigroup $(T(t))_{t\ge 0}$ on $X$.
For systems described by (time-dependent) PDEs and with controls and observations (measurements) acting on the spatial boundary, the operators $B$ and $C$ become ``unbounded'' with respect to the state space $X$, in the sense that  only $B\in \calL(U,X_{-1})$ and $C\in \calL(X_1,Y)$; where $X_{-1}$ denotes the extrapolation space associated to $(T(t))_{t\ge 0}$ and $X_1$ denotes the interpolation space $\dom( A)$. This approach is explained in \cite{EmirsjlowTownley2000, Salamon1984, Salamon1987}; see also \cite{TucsnakWeiss2009,Schwenninger2020,Staffans2005}.

For each $x_0\in X$, the system $\Sigma(A,B,0)$ has a mild solution in $X_{-1}$ given by 
\[
    x(t) = T(t)x_0 + \int_0^t T_{-1}(t-s)Bu(s)\dx s;
\]
where $(T_{-1}(t))_{t\ge 0}$ denotes the extrapolated semigroup on $X_{-1}$. In this case, it makes sense to ask whether the solution lies in $X$, which gives rise to the notion of the admissibility of control operators. Let $\mathrm Z$ be a placeholder for $\mathrm C$ (to denote continuous functions) or $L^p$ with $p\in [1,\infty]$. We say that $B$ is a \emph{$\mathrm Z$-admissible control operator} if  for some (equivalently, all) $\tau>0$, the \emph{input map} -- defined as
\begin{equation}
    \label{eq:input-operator}
        \Phi_{\tau}:\mathrm{Z}([0,\tau],U)\to       X_{-1},\qquad
                                        u \mapsto   \int_0^{\tau}T_{-1}(\tau-s)Bu(s)\dx s
\end{equation}
satisfies $\Ima \Phi_{\tau}\subseteq X$. Correspondingly, the solution of $\Sigma(A,0,C)$ is given by
\[
    y(t) = CT(t)x_0,\qquad (x_0\in X_1).
\]
We say that $C$ is a \emph{$\mathrm Z$-admissible observation operator} if the \emph{output map} 
\begin{equation}
    \label{eq:output-operator}
        \Psi_{\tau}:X_1 \to     \mathrm{Z}([0,\tau],Y),\qquad
                    x   \mapsto   CT(\argument)x
\end{equation}
has a bounded extension to $X$ for some (equivalently, all) $\tau>0$.

 The concept of admissible operators is fundamental to the investigation of infinite-dimensional systems, and their significance is especially pronounced in the realm of well-posed systems, where they facilitate the establishment of stability, controllability, and observability of such systems \cite{JacobPartington2004, Staffans2005, TucsnakWeiss2014}. The theory of admissibility for the case $\mathrm Z=L^2$ and $X$ being a Hilbert space is classical \cite{TucsnakWeiss2009, JacobPartington2001, Staffans2005}. The case $p\in(1,\infty)$
 has also garnered significant attention \cite{HaakLeMerdy2005, Haak2004, JacobPartingtonPott2014, Weiss1989a, Weiss1989b}. In this context, the Weiss duality result \cite[Theorem~6.9]{Weiss1989a} plays an important role. In particular, under the assumption that the dual semigroup $(T'(t))_{t\ge 0}$ is strongly continuous and $p,q\in (1,\infty)$ are H\"older conjugates, it says that $B$ is a $L^p$-admissible control operator if and only if $B'$ is a $L^q$-admissible observation operator and analogously, for $C$'s. More recently, there is growing interest in ``limit-case'' admissibility \cite{JacobNabiullinPartingtonSchwenninger2018, JacobSchwenningerZwart2019, MironchenkoPrieur2020, AroraGlueckPaunonenSchwenninger2024,PreusslerSchwenninger2024}, referring in particular to $L^{\infty}$-admissible control operators because of their importance in the study of input-to-state stability (ISS); see \cite{JacobNabiullinPartingtonSchwenninger2018,MironchenkoPrieur2020}.

The Weiss duality result enables the translation of various (negative) results between control and observation operators, especially where the state space is reflexive. In practice, however, there are multiple situations where the dual semigroup $(T'(t))_{t\ge 0}$ is not strongly continuous on $X'$ or where $X$ has no pre-dual, for instance, where $X$ is an $L^1$-space -- which is often the case when studying $L^1$-admissibility for observation operators. As a result, various facts known for control operators cannot be translated to the observation operators and vice versa. An important example here is \cite[Theorem~4.8]{Weiss1989b} which says that if $X$ is reflexive, then $B$ is a $L^1$-admissible control operator if and only if $B\in \calL(U, X)$. The reflexivity of $X$ cannot be dropped as is shown in \cite[Negative result~5.4]{Weiss1989b} by taking a periodic left shift semigroup on $L^1\big([0,2\pi]\big)$. Since $L^1\big([0,2\pi]\big)$ does not have a predual, the same example cannot be used to show the existence of an unbounded $L^\infty$-admissible observation operator.  Similarly, the fact that all $L^1$-admissible control operators are those that map into the Favard space associated with the extrapolated semigroup \cite[Corollary~17]{MaraghBounitFadiliHammouri2014} cannot be dualized if the dual semigroup lacks strong continuity.

In operator semigroups, the classical approach to circumvent the above issue of strong continuity is restricting the dual semigroup to the closed subspace on which the dual semigroup is strongly continuous; the {\emph{sun-dual space}}. Remarkably, this (still) allows for a rich theory, mostly developed in the 1980s, see \cite{ClementDiekmannGyllenbergHeijmansThieme1987,ClementDiekmannGyllenbergHeijmansThieme1988,ClementDiekmannGyllenbergHeijmansThieme1989a,ClementDiekmannGyllenbergHeijmansThieme1989b,DiekmannGyllenbergThieme1991}, as well as the monograph by van Neerven \cite{vanNeerven1992}. These works take motivation ranging from classical age population models over delay equations \cite{diekmanvgils1991,diekmannvangils1995} to models arising in neuroscience \cite{Spe2020}, where $L^{1}$- and $\sup$-norms are naturally appearing. However, in the context of admissible operators and more generally infinite-dimensional systems theory, sun-duality has hardly been employed; see \cite{controlsundual2} for controllability results and \cite{controlsundual1} for some optimal control problems on non-reflexive spaces. This is the gap we would like to close in the present paper. Our original motivation for this lies in characterising $\mathrm{Z}$-admissible operators, particularly for semigroups with a non-trivial sun-dual. Let us showcase why this is of interest:~it is still an open question whether $L^{\infty}$-admissible control operators are always \emph{zero-class}, i.e., whether $\lim_{\tau\to0^{+}}\|\Phi_{\tau}\|_{\mathcal{L}(L^{\infty}([0,\tau],U),X)}=0$, with $\Phi_\tau$ defined in~\eqref{eq:input-operator}; see for instance, \cite[Section~6]{JacobNabiullinPartingtonSchwenninger2018}. On the other hand, the formally dual question can be answered in the negative \cite[Example~26]{JacobSchwenningerZwart2019}:~there exists $L^{1}$-admissible observation operators such that $\lim_{\tau\to0^{+}}\|\Psi_{\tau}\|_{\mathcal{L}(X,L^{1}([0,\tau],Y))}\neq0$. It is not possible to link these two settings by the usual duality as the involved function spaces are $L^{1}$-spaces and $L^{\infty}$-spaces. Moreover, by Lotz's result \cite{lotz1985},  any strongly continuous semigroup on $L^{\infty}\left([0,1]\right)$ -- the dual of the state space of the mentioned counterexample -- has a bounded generator, which readily implies zero-class admissibility. In \cite{JacobSchwenningerWintermayr2022}, it was indeed shown that $L^{\infty}$-admissibility of $B=A_{-1}$, the extension of $A$ to an operator from $X$ to $X_{-1}$, implies that $A$ is bounded, resting on deep results from the geometry of Banach spaces and a connection to maximal regularity for parabolic equations. Our results show that the sun-duality is the right framework to dualise these situations; in particular we show in Theorem \ref{thm:sun-dual-continuous} that $\mathrm{C}$-admissibility of control operators $B$ is the proper dual concept for $B'$ being an $L^{1}$-admissible observation operator with respect to the sun-dual semigroup. 

Zero-class admissible observation operators first appeared in \cite{XuLiuYung2008} in the context of observability. We refer to \cite{HaakOuhabaz2012} for details and  more general results.
%
%
On the other hand, for $p<\infty$, zero-class $L^p$-admissible control operators ensures continuity of the mild solution with values in the state space \cite[Proposition~2.3]{Weiss1989b}. 

We note the connection of admissible operators to perturbation theory for operator semigroups, given by the classical Miyadera-Voigt and Desch-Schappacher theorems, see, for example, \cite[Chapter~3]{EngelNagel2000}. In the system-theoretic context described above, these can in essence be phrased as follows:~if  perturbations $C\in\mathcal{L}(X_{1},X)$ or $B\in \mathcal{L}(X,X_{-1})$ are zero-class $L^{1}$- or $\mathrm{C}$-admissible, respectively, then the perturbed semigroup $A+C$ or the part of $A_{-1}+B$ in $X$, respectively, generate $C_{0}$-semigroups.  It is worth mentioning that the sun-dual theory \cite{ClementDiekmannGyllenbergHeijmansThieme1987} originated from perturbation results around the same time.  More precisely, in \cite{ClementDiekmannGyllenbergHeijmansThieme1987}, see also \cite[Theorems~3.2.6 and~4.3.5]{vanNeerven1992}, it was shown that if $B\in\mathcal{L}(X,X^{\sun\times})$, then the part of $A_{-1}+B$ in $X$ generates a $C_{0}$-semigroup, where the space $X^{\sun\times}$ can be isomorphically identified with the Favard space of the extrapolated semigroup on $X_{-1}$. We skip the definitions of those spaces but point out that $\mathcal{L}(X,X^{\sun\times})$ is isomorphic to the set of $L^1$-admissible control operators from $X$ to $X_{-1}$, in  \cite[Corollary~17]{MaraghBounitFadiliHammouri2014}.

Our duality results are given in Sections~\ref{sec:control} and~\ref{sec:observation}, generalising the duality result by Weiss from \cite[Theorem 6.9]{Weiss1989a}, dropping any condition of the form $X^{\sun}=X'$. For convenience, we  summarise the scenario in Figures~\ref{figure:control} and~\ref{figure:observation}. The penultimate Section~\ref{sec:example} is devoted to a prototypical example for which the limit-case admissibility is characterised.
The article concludes with some remarks in Section~\ref{sec:conclusion}.

\begin{figure}%
    \begin{tikzpicture}[>=implies,thick, minimum height=0.5cm, minimum width=2cm]
        \node (C) at (0,0) {$B$ is $\mathrm{C}$-admissible for $(T(t))_{t\ge 0}$};
    
        \node (1) at (7,0) {$B'$ is $L^{1}$-admissible for $(T^{\sun}(t))_{t\ge 0}$};
    
        \node (p) at (0,-2) {$B$ is $L^{p}$-admissible for $(T(t))_{t\ge 0}$};
    
        \node (q) at (7,-2) {$B'$ is $L^{q}$-admissible for $(T^{\sun}(t))_{t\ge 0}$};
    
        \draw[thick,double equal sign distance,<->] (4.3,0) to (2.5,0);
    
        \draw[thick,double equal sign distance,->] (-0.5,-1.7) to (-0.5,-0.25);
    
        \draw[thick,double equal sign distance,->] (6.5,-1.7) to (6.5,-0.25);
    
        \draw[thick,double equal sign distance,<-] (4.3,-1.9) to (2.5,-1.9);
        
        \draw[purple,thick,double equal sign distance,->] (4.3,-2.2) to (2.5,-2.2);
    
        \node[purple] (zero) at (3.4,-2.5) {\small $p<\infty$};
    \end{tikzpicture}
    \caption{Duality between control operators $B\in \calL(U, X_{-1})$ and observation operators $B'\in \calL\left((X^{\sun})_{1},U'\right)$}%
    \label{figure:control}
\end{figure}

\begin{figure}%
 \begin{tikzpicture}[>=implies,thick, minimum height=0.5cm, minimum width=2cm]
        \node (C) at (0,0) {$C$ is $L^1$-admissible for $(T(t))_{t\ge 0}$};
    
        \node (1) at (7,0) {$C'$ is $\mathrm{C}$-admissible for $(T^{\sun}(t))_{t\ge 0}$};
    
        \node (p) at (0,-2) {$C$ is $L^{p}$-admissible for $(T(t))_{t\ge 0}$};
    
        \node (q) at (7,-2) {$C'$ is $L^{q}$-admissible for $(T^{\sun}(t))_{t\ge 0}$};
    
         \draw[purple,thick,double equal sign distance,<-] (4.3,0.1) to (2.5,0.1);
        
        \draw[thick,double equal sign distance,->] (4.3,-0.2) to (2.5,-0.2);

        \draw[thick,double equal sign distance,->] (-0.5,-1.7) to (-0.5,-0.25);

        \draw[purple,thick,double equal sign distance,<-] (4.3,-1.9) to (2.5,-1.9);

        \draw[thick,double equal sign distance,->] (4.3,-2.2) to (2.5,-2.2);

        \draw[thick,double equal sign distance,->] (6.5,-1.7) to (6.5,-0.25);
    
        \node[purple] (zero) at (3.4,0.4) {\small zero-class};

        \node[purple] (p) at (3.3,-1.6) {\small $p>1$};
    \end{tikzpicture}
    \caption{Duality for observation operators $C\in \calL(X_1, Y)$ with $\Ima C'  \subseteq \left(X^{\sun}\right)_{-1}$ and control operators $B'\in \calL\left(Y',\left(X^{\sun}\right)_{-1}\right)$.}%
    \label{figure:observation}
\end{figure}

\subsection*{Preliminaries}

Let $(T(t))_{t\ge 0}$ be a $C_0$-semigroup on a Banach space $X$.
We use the notation $X^{\sun}$, to denote the subspace of $X'$ where the dual semigroup $(T(t)')_{t\ge 0}$ is strongly continuous. The restricted $C_0$-semigroup is as usual denoted by $(T^{\sun} (t))_{t\ge 0}$.
For the theory of sun-dual semigroups, we refer the reader to \cite{vanNeerven1992}. 

Let $U$ and $Y$ be Banach spaces and let $\mathrm Z$ be a placeholder for $\mathrm C$ or $L^p$. For $B\in \calL(U,X_{-1})$, we say that $B$ is a \emph{zero-class $\mathrm Z$-admissible} control operator if the input map in~\eqref{eq:input-operator} satisfies $\lim_{\tau\downarrow 0}\norm{\Phi_\tau}_{\calL(\mathrm Z([0,\tau],U),X)}=0$. Likewise,  $C\in \calL(X_1,Y)$ is called a \emph{zero-class $\mathrm Z$-admissible} observation operator if the output map in~\eqref{eq:output-operator} fulfils $\lim_{\tau\downarrow 0}\norm{\Psi_\tau}_{\calL(X,\mathrm Z([0,\tau],Y))}=0$. 
For $p\in [1,\infty]$, we write $\bbC_p(X, Y, (T(t))_{t\ge 0})$ for the subspace of $\calL(X_1,Y)$ of all $L^p$-admissible observation operators and set
\[
    \norm{C}_{\bbC_p(X, Y, \tau)}:= \norm{\Psi_\tau}_{\calL(X,L^p([0,\tau],Y))}.
\]
Similarly, $\bbB_p(U, X, (T(t))_{t\ge 0}))$ denotes the subspace of $\calL(U, X_{-1})$ of all $L^p$-admissible control operators with
\[
    \norm{B}_{\bbB_p(U, X, \tau)}:=\norm{\Phi_\tau}_{\calL( L^p([0,\tau],U),X)}
\]
For convenience, the notation $\bbB_{\mathrm C}(U, X, (T(t))_{t\ge 0}))$ is sometimes used to denote the $\mathrm C$-admissible control operators with
\[
    \norm{B}_{\bbB_{\mathrm C}(U, X, \tau)}:=\norm{\Phi_\tau}_{\calL(\mathrm C([0,\tau],U),X)}
\]
denoting the corresponding norm.

\section{Characterisation of $\mathrm C$-admissibility of control operators}

Let $U$ be a Banach space and denote by $\mathrm{T}([0,\tau], U)$, the space of all $U$-valued step functions on $[0,\tau]$, i.e., piecewise constant functions with finitely many pieces. Equipped with the supremum norm, $\mathrm{T}([0,\tau], U)$ becomes a normed space whose completion is the space of \emph{regulated} functions $\Reg([0,\tau], U)$. One can therefore define $\Reg$-admissibility by replacing $\mathrm Z$ by $\Reg$ in~\eqref{eq:input-operator}. Since every continuous function is regulated, it is immediate that $\Reg$-admissibility implies $\mathrm C$-admissibility. Actually, the two notions are even equivalent \cite[Proposition~4.2]{AroraGlueckPaunonenSchwenninger2024}. The following result, which is an extension of \cite[Theorem~10.2.2]{Staffans2005} -- adapting an argument of Travis \cite[Lemma~3.1 and Proposition~3.1]{Travis1981}, see also \cite[Proposition~2.2]{JacobSchwenningerWintermayr2022} --  characterizes the class of all $\mathrm C$-admissible control operators.

Let $X, \calX$, and $U$ be Banach spaces. Recall that the \emph{semivariation} of a function $F:[0,\tau]\to \calL(U,X)$ is defined as
\[
    \SV_0^{\tau}(F):= \sup_{\substack {\norm{u_i}_U \le 1\\ 0=t_1<t_2<\ldots<t_n=\tau \\ n\in \bbN}}\norm{\sum_{i=1}^{n} \big(F(t_i)-F(t_{i-1})\big)u_i  }_X 
\]
and $F$ is said to be of \emph{bounded semivariation} on $[0,\tau]$ if $\SV_0^{\tau}(F)<\infty$. Moreover, the \emph{variation} of a function $f:[0,\tau]\to \calX$ is given by
\[
    \Var_0^{\tau}(f):=\sup_{\substack { 0=t_1<t_2<\ldots<t_n=\tau \\ n\in \bbN}}  \sum_{i=1}^{n} \norm{f(t_i)-f(t_{i-1})  }_{\calX} 
\]
and $f$ is said to have \emph{bounded variation} on $[0,\tau]$ if $\Var_0^{\tau}(f)<\infty$. 
Observe that $\SV_0^{\tau}(F)\le \Var_0^{\tau}(F)$.
A thorough treatment of functions of bounded variation can be found in \cite{AmbrosioFuscoPallara2000}, see also \cite[Section~3.2]{HillePhillips1974} and for functions of bounded semivariation, we refer to the survey \cite{Monteiro2015}.

\begin{proposition}
    \label{prop:reg-admissibility-equivalence}
    Let $X$ and $U$ be Banach spaces, $(T(t))_{t\ge 0}$ be a $C_0$-semigroup on $X$ with generator $A$, let $\lambda \in \resSet(A)$, and let $\tau>0$.
    For the control operator $B\in\calL(U,X_{-1})$, the following are equivalent.
    \begin{enumerate}[\upshape (i)]
        \item The control operator $B$ is $\mathrm C$-admissible.
        \item The function $T(\argument)\Res(\lambda, A_{-1})B$ is of bounded semivariation on $[0,\tau]$.
        \item For each $x' \in X'$ with $\norm{x'}\le 1$, the function $B' \Res(\overline{\lambda}, A')T(\argument)'x'$ is of bounded variation on $[0,\tau]$.
        \item The control operator $B$ is $\Reg$-admissible.
    \end{enumerate}
    Moreover, setting $F(\argument):=T(\argument)\Res(\lambda,A_{-1})B$, we have that  
    \begin{equation}
        \label{eq:variation-inequality-i}
        \big(1-\norm{\lambda \Res(\lambda, A)}\big)\norm{B}_{\bbB_{\mathrm C}(U, X, \tau)} \le \Var_0^{\tau}(F(\argument)'x')  
    \end{equation}
    and
    \begin{equation}
        \label{eq:variation-inequality-ii}
        \Var_0^{\tau}(F(\argument)'x') \le \SV_0^{\tau}(F) \le \norm{1- \lambda\Res(\lambda,A)} \norm{B}_{\bbB_{\mathrm C}(U, X, \tau)}
    \end{equation}
    for all $x' \in X'$ with $\norm{x'}\le 1$.
\end{proposition}

\begin{proof}
    The equivalence of (iii) and (iv) and the inequality in~\eqref{eq:variation-inequality-i} is proved in \cite[Theorem~10.2.2]{Staffans2005}, whereas implication (iv) $\Rightarrow$ (i) is obvious.

    ``(i) $\Rightarrow$ (ii)'': The proof is along the lines of \cite[Proposition~3.1]{Travis1981}. We include the details for convenience of the reader.
    Let $B$ be $\mathrm C$-admissible and for simplicity, suppose $\tau=1$.
     Consider a partition $0=t_0<t_1<\ldots<t_n=\tau$ of $[0,1]$ and let $\epsilon< \min \modulus{t_i-t_{i-1}}$. Fix arbitrary elements $u_1,\ldots,u_{n+1}\in U$ with $\norm{u_i}\le 1$, define $u_{\epsilon}:[0,1]\to X$ as
    \[
        u_{\epsilon}(s) := \begin{cases}
                            u_i,                                            \qquad & t_{i-1}\le s\le t_i-\epsilon\\
                            u_{i+1}+(u_{i+1}-u_i)\frac{s-t_i}{\epsilon},    \qquad & t_i-\epsilon\le s\le t_i.
        \end{cases} 
    \]
    Since $u_{\epsilon}$ is continuous, admissibility implies that $\Phi_{1}u_{\epsilon}\in X$. 
    For simplicity, we set $w_i=\Res(\lambda,A_{-1})Bu_i \in \dom(A_{-1})= X$ for $1\le i\le n_1$.
    Then we can write
    \begin{align*}
        (\lambda \Res(\lambda,A)-1)\Phi_1 u_{\epsilon} &= A\Res(\lambda,A)\Phi_1 u_{\epsilon}\\
                                                    & = A \int_0^{1} T(1-s)\Res(\lambda,A_{-1})Bu_{\epsilon}(s)\dx s\\
                                                    &=  A \left[\sum_{i=1}^n \int_{t_{i-1}}^{t_i-\epsilon} T(1-s)w_i\dx s\right.\\
                                                    &\qquad\quad+ \left.\int_{t_i-\epsilon}^{t_i} T(1-s)(w_{i+1}+(w_{i+1}-w_i)\frac{s-t_i}{\epsilon}  \dx s\right]   
    \end{align*}
    
    Now, we repeat the computations in the proof of \cite[Proposition~3.1]{Travis1981}:~from the last expression, $(\lambda \Res(\lambda,A)-1)\Phi_1 u_{\epsilon}$ simplifies to
    \begin{align*}
    &\sum_{i=1}^n\Big[ \big(T(1-t_{i-1}) - T(1-t_i+\epsilon)\big)w_i
                                                         + \big(T(1-t_{i}+\epsilon) - T(1-t_i)\big)w_{i+1}\Big]\\
                                                        &\qquad+ \sum_{i=1}^n \left[  \frac{1}{\epsilon} \int_{t_i-\epsilon}^{t_i} T(1-s)(w_{i+1}-w_i)\dx s - T(1-t_i+\epsilon)(w_{i+1}-w_i)\right]\\
    = &  - \sum_{i=1}^n\big(T(1-t_i)-T(1-t_{i-1})\big) w_i \\
                    &\qquad+\sum_{i=1}^n \left[\frac{1}{\epsilon} \int_{t_i-\epsilon}^{t_i} T(1-s)(w_{i+1}-w_i)\dx s - T(1-t_i)(w_{i+1}-w_i)\right].
    \end{align*}
    For this reason, $\norm{ \sum_{i=1}^n \big(T(1-t_i)-T(1-t_{i-1})\big)\Res(\lambda,A_{-1}) B u_i }$ can be estimated from above by
    \begin{align*}
         \norm{(\lambda \Res(\lambda,A)-1)\Phi_1} + \sum_{i=1}^n \norm{\frac{1}{\epsilon} \int_{t_i-\epsilon}^{t_i} T(1-s)(w_{i+1}-w_i)\dx s - T(1-t_{i})(w_{i+1}-w_i)   } 
    \end{align*}
    Taking $\epsilon\to 0$ yields that $F$ is of bounded semivariation and 
    \[
        \SV_0^{\tau}(F) \le \norm{1- \lambda\Res(\lambda,A)} \norm{B}_{\bbB_{\mathrm C}(U, X, \tau)}.
    \]

    ``(ii) $\Rightarrow$ (iii)'': Let $x' \in X'$ with $\norm{x'}\le 1$.
    We need to show that $F(\argument)'x'$ is of bounded variation.
    Consider a partition $0=t_0<t_1<\ldots<t_n=\tau$ of $[0,\tau]$. Fix $\epsilon \in (0,1)$ and for each $1\le i\le n$, choose $u_i \in U$ with $\norm{u_i}\le 1$ such that
    \begin{align*}
        (1-\epsilon)\norm{ F(t_i)' x'-F(t_{i-1})'x' }_{U'}&\le \duality{F(t_i)' x'-F(t_{i-1})'x'}{u_i}  \\
                                                    &=\duality{x'}{\big(F(t_i)-F(t_{i-1})\big)u_i}.  
    \end{align*}
    Employing bounded semivariation of $F(\argument)$ together with
    \begin{align*}
        \sum_{i=1}^n \norm{F(t_i)'x'-F(t_{i-1})'x'}_{U'} &\le \frac{1}{(1-\epsilon)}\duality{x'}{\sum_{i=1}^n\big(F(t_i)-F(t_{i-1})\big)u_i} \\
                                                         &\le \frac{1}{(1-\epsilon)}\norm{\sum_{i=1}^n\big(F(t_i)-F(t_{i-1})\big)u_i  }
    \end{align*}
     yields bounded variation of $F(\argument)'x'$. Since $\epsilon>0$ was arbitrary, we also get ${\Var_0^{\tau}(F(\argument)'x')\le \SV_0^{\tau}(F)}$.
\end{proof}

\section{Duality results for control operators}
    \label{sec:control}

In \cite[Theorem~6.9]{Weiss1989a}, Weiss explored the dual relationship between $L^p$-admissible observation operators and $L^q$-admissible control operators for Hölder conjugates $p$ and $q$. The result, however, assumes strong continuity of the dual semigroup. Restricting to the sun-dual space -- the space of strong continuity of the dual semigroup -- it is natural to ask whether \cite[Theorem~6.9]{Weiss1989a} can be appropriately generalised. We explore this for control operators in the present section. 

In our first result, we show that $\mathrm C$-admissibility of the control is equivalent to $L^1$-admissibility of the dual observation operator. Keeping Proposition~\ref{prop:reg-admissibility-equivalence} in mind, the proof of the necessity in the reflexive case was given in \cite[Theorem~10.2.2]{Staffans2005} and the converse for the case $X^{\sun}=X'$ was indicated in \cite[Remark~6.10]{Weiss1989a}.

\begin{theorem}
    \label{thm:sun-dual-continuous}
    Let $X$ and $U$ be Banach spaces and let $(T(t))_{t\ge 0}$ be a $C_0$-semigroup on $X$ with generator $A$. 
    
    A control operator $B\in \calL(U, X_{-1})$ is (zero-class) $\mathrm C$-admissible if and only if the observation operator $B' \in \calL\left((X^{\sun})_1, U'\right)$ is (zero-class) $L^1$-admissible.
\end{theorem}

We point out that the observation operator considered above is actually the restriction of $B': (X_{-1})'\to U'$ to the interpolation space $(X_{-1})^{\sun}=(X^{\sun})_1$. For this reason, while our result may seem like a straightforward generalisation, we emphasise that the invariance of the zero-class property is not necessarily expected.

\begin{proof}[Proof of Theorem~\ref{thm:sun-dual-continuous}]
    First, assume that $B\in \calL(U, X_{-1})$ is $\mathrm C$-admissible with input operator $\Phi_{\tau}$.
    Fix $\lambda \in \resSet(A)$ and $x^{\sun}\in \dom(A^{\sun})$.
    Setting $F(\argument):=B' \Res(\overline{\lambda}, A')T^{\sun}(\argument)$, the continuity of $s\mapsto B'T^{\sun}(s)x^{\sun}$ allows us to estimate
    \begin{align*}
        \int_0^{\tau} \norm{F(s)x^{\sun}}\dx s & =  \int_0^{\tau} \norm{F(0)x^{\sun}+\int_0^s \frac{d}{dr}F(r)x^{\sun}\dx r}\dx s \\
                                               & \le \int_0^{\tau} \norm{F(0)x^{\sun}}\dx s + \int_{0}^\tau \int_0^s \norm{\frac{d}{dr}F(r)x^{\sun}}\dx r\dx s\\
                                               & = \int_0^{\tau} \norm{F(0)x^{\sun}}\dx s + \int_{0}^\tau \int_r^{\tau} \norm{\frac{d}{dr}F(r)x^{\sun}}\dx s\dx r\\
                                               & = \int_0^{\tau} \norm{F(0)x^{\sun}}\dx s + \int_{0}^\tau (\tau-r) \norm{\frac{d}{dr}F(r)x^{\sun}}\dx r\\
                                               & \le \tau \norm{F(0)x^{\sun}} + \tau \int_{0}^\tau  \norm{\frac{d}{dr}F(r)x^{\sun}}\dx r;
    \end{align*}
    where the second equality is obtained using Fubini. As a result,
    \begin{align*}
        \int_0^{\tau} \norm{B'T^{\sun}(s)x^{\sun}} \dx s 
                                                        & = \int_0^{\tau} F(s)(\overline{\lambda}-A')x^{\sun}\dx s\\
                                                        & \le \modulus{\lambda} \int_0^{\tau} \norm{F(s)x^{\sun}}\dx s + \int_0^{\tau} \norm{\frac{d}{ds}F(s)x^{\sun}}\dx s\\
                                                        & \le \tau\modulus{\lambda} \norm{F(0)}\norm{x^{\sun}} + (\tau\modulus{\lambda}+1)\int_{0}^\tau  \norm{\frac{d}{dr}F(s)x^{\sun}}\dx s\\
                                                        & = \tau\modulus{\lambda} \norm{F(0)}\norm{x^{\sun}} + (\tau\modulus{\lambda}+1)\Var_0^{\tau}\left(F(\argument)x^{\sun}\right)
    \end{align*}
    because $\int_0^{\tau} \norm{\frac{d}{ds}F(s)x^{\sun}  }\dx s$ is the total variation of $F(\argument)x^{\sun}$ on $[0,\tau]$.
    Due to $\mathrm C$-admissibility of $B$, we can now apply Proposition~\ref{prop:reg-admissibility-equivalence}(iii) -- and, in particular, the inequality~\eqref{eq:variation-inequality-ii} -- to obtain that
    \begin{equation}
        \label{eq:c-to-L1}
        \int_0^{\tau} \norm{B'T^{\sun}(s)x^{\sun}} \dx s \le \tau\modulus{\lambda} \norm{F(0)}\norm{x^{\sun}} + C_{\tau} \norm{B}_{\bbB_{\mathrm C}(U, X, \tau)}\norm{x^{\sun}}
    \end{equation}
    with $C_{\tau}:=2(\tau\modulus{\lambda}+1)\norm{1-\lambda \Res(\lambda,A)}$.
    It follows that $B' \in \bbC_1(X^{\sun}, U',(T^{\sun}(t))_{t\ge 0})$.

    Conversely, let $B' \in \bbC_1(X^{\sun}, U',(T^{\sun}(t))_{t\ge 0})$. We show that $B$ is $\Reg$-admissible. First of all, for $u \in \mathrm{T}([0,\tau], U)$ -- the space of $U$-valued step functions, of course
    \[
            \Phi_{\tau}u:= \int_0^{\tau} T_{-1}(\tau-t) Bu(t)\dx t \in X.
    \]
    By density of step functions in regulated functions, we therefore only need to show that there exists $K>0$ such that $\norm{\Phi_{\tau}u}_X \le K \norm{u}_{\infty}$ for all $u \in \mathrm{T}([0,\tau], U)$.

    To this end, fix $x^{\sun} \in \dom(A^{\sun})$, and
    for each $u \in \mathrm{T}([0,\tau], U)$ estimate
    \begin{align*}
        \modulus{\duality{\Phi_\tau u}{x^{\sun}}} & =  \modulus{ \duality{\int_0^{\tau} T_{-1}(\tau-t) Bu(t)\dx t}{x^{\sun}}  }\\
                                                    & = \modulus{ \int_0^{\tau} \duality{u(t)}{B'(T_{-1})^{\sun}(\tau-t)x^{\sun}} \dx t  }\\
                                                    & = \modulus{ \int_0^{\tau} \duality{u(t)}{B' T^{\sun}(\tau-t)x^{\sun}} \dx t  }\\
                                                    & \le \norm{B'}_{\bbC_1(X^{\sun}, U',\tau   )}\norm{x^{\sun}}\norm{u}_{\infty}.
    \end{align*}
    From the density of $\dom(A^{\sun})$ in $X^{\sun}$, we infer that the above inequality also holds for each $x^{\sun}\in X^{\sun}$. 
     Consequently, the norming property of the sun-dual \cite[Theorem~1.3.5]{vanNeerven1992} yields the desired estimate:
    \begin{equation}
        \label{eq:l1-to-c}
        \norm{\Phi_{\tau}u}_X \le \norm{B'}_{\bbC_1(X^{\sun}, U',\tau   )}\limsup_{t\downarrow 0}{\norm{T(t)}} \norm{u}_\infty.
    \end{equation}
    
    Lastly, the zero-class equivalence is immediate from estimates~\eqref{eq:c-to-L1} and~\eqref{eq:l1-to-c}.
\end{proof}

Next, we generalise \cite[Theorem~6.9(ii)]{Weiss1989a} to go from $L^p$-admissibility of the control to $L^q$-admissibility of its dual, where $\frac{1}{p}+\frac{1}{q}=1$.

\begin{theorem}
    \label{thm:weiss-sun-dual-control}
    Let $X$ and $U$ be Banach spaces, let $(T(t))_{t\ge 0}$ be a $C_0$-semigroup on $X$ with generator $A$, and 
    let $B\in \calL(U, X_{-1})$.
    
    For Hölder conjugates $p,q\in [1,\infty]$, if
    $B \in \bbB_p(U, X, (T(t))_{t\ge 0}))$,
    then $B' \in \bbC_q(X^{\sun},U',(T^{\sun}(t))_{t\ge 0})$ with 
    \[
        \norm{B'}_{\bbC_q(X^{\sun},U',\tau)} \le \norm{B}_{\bbB_p(U, X, \tau)}.
    \]
    The converse is true if $p<\infty$.
\end{theorem}

\begin{remarks}
    (a) 
    The condition $p<\infty$ cannot be dropped in the converse part of Theorem~\ref{thm:weiss-sun-dual-control}; see \cite[Remark~2.4]{JacobSchwenningerWintermayr2022}. The example in the reference -- which has also appeared in the context of maximal regularity (see, \cite[Page~48]{EberhardtGreiner1992} and \cite[Example~2.3]{JacobSchwenningerWintermayr2022}) and admissibility \cite[Remark~4.8 and Page~26]{AroraGlueckPaunonenSchwenninger2024} -- satisfies $X^{\sun} = X'$.
    
    (b)
    Actually, if $X$ is reflexive, then $L^1$-admissibility of the dual of a control operator $B$ does imply $L^\infty$-admissibility of $B$ \cite[Theorem~6.9(ii)]{Weiss1989a}. 
    This begs the questions whether sun-reflexivity of the semigroup is a sufficient condition to obtain the converse in Theorem~\ref{thm:weiss-sun-dual-control} for the case $p=\infty$. However, an evidence to the contrary is again provided by the example in \cite[Remark~2.4]{JacobSchwenningerWintermayr2022}.
    
\end{remarks}

\begin{proof}[Proof of Theorem~\ref{thm:weiss-sun-dual-control}]
    Let $B \in \bbB_p(U, X, (T(t))_{t\ge 0}))$. By the norming property of Banach space valued $L^p$-spaces \cite[Proposition~1.3.1]{HytoenenVanNeervenVeraarWeis2016}, we can compute the norm
    \begin{align*}
        \norm{B'T^{\sun}(\argument)x^{\sun}}_q & = \sup_{\substack{ u \in L^p([0,\tau], U)\\ \norm{u}_p\le 1}} \modulus{\int_0^{\tau} \duality{u(t)}{B'T^{\sun}(\tau-t)x^{\sun}}\dx t}\\
                                                    & = \sup_{\substack{ u \in L^p([0,\tau], U)\\ \norm{u}_p\le 1}} \modulus{\int_0^{\tau} \duality{u(t)}{B'(T_{-1})^{\sun}(\tau-t)x^{\sun}}\dx t}\\
                                                    & = \sup_{\substack{ u \in L^p([0,\tau], U)\\ \norm{u}_p\le 1}} \modulus{ \duality{\int_0^{\tau}T_{-1}(\tau-t)Bu(t)\dx t}{x^{\sun}}}\\
                                                    & \le \norm{B}_{\bbB_p(U, X,\tau)} \norm{x^{\sun}}
    \end{align*}
    for all $x^{\sun} \in \dom(A^{\sun})$ and so $B' \in \bbC_q(X^{\sun},U',(T^{\sun}(t))_{t\ge 0})$.

    Conversely, let $B' \in \bbC_q(X^{\sun},U',(T^{\sun}(t))_{t\ge 0})$. Employing Hölder's inequality, we can argue exactly as in  Theorem~\ref{thm:sun-dual-continuous}, to obtain a constant $K>0$ such that
    $\norm{\Phi_{\tau}u} \le K\norm{u}_p$ for all step functions $u\in \mathrm{T}([0,\tau], U)$. If $p<\infty$, this implies that $B$ is $L^p$-admissible
    by density of the step functions in $L^{p}([0,\tau],U)$.
\end{proof}

\section{Duality results for observation operators}
    \label{sec:observation}

In this section, we look at dual of observation operators, i.e., analogous to \cite[Theorem~6.9(i)]{Weiss1989a}, we ask whether the equivalence 
\begin{equation}
    \label{eq:admissibility-observation-sun-dual}
    C \in \bbC_p(X, Y, (T(t))_{t\ge 0}) \stackrel{?}{\iff} C'\in \bbB_q(Y', X^{\sun}, (T^{\sun}(t))_{t\ge 0})
\end{equation}
holds for Hölder conjugates $p,q\in [1,\infty]$. Note that given $C\in \calL(X_1,Y)$, we only know that $\Ima C'\subseteq (X_1)'=(X')_{-1}$. Therefore, in order for the second inclusion in~\eqref{eq:admissibility-observation-sun-dual} to be meaningful, we must a priori assume that $\Ima C'  \subseteq \left(X^{\sun}\right)_{-1}$. Adapting the arguments of \cite[Theorem~6.9]{Weiss1989a}, we first settle the reverse implication in~\eqref{eq:admissibility-observation-sun-dual}:

\begin{theorem}
    \label{thm:weiss-sun-dual-observation-necessary}
    Let $X$ and $Y$ be Banach spaces, let $(T(t))_{t\ge 0}$ be a $C_0$-semigroup on $X$ with generator $A$, and 
    let $C\in \calL(X_1,Y)$ be such that $C'(Y')\subseteq (X_1)^{\sun}= \left(X^{\sun}\right)_{-1}$. 
    
    Let $p,q\in [1,\infty]$ be Hölder conjugates. If $C'\in \bbB_q(Y', X^{\sun}, (T^{\sun}(t))_{t\ge 0})$, then $C \in \bbC_p(X, Y, (T(t))_{t\ge 0})$ with
    $
        \norm{C}_{\bbC_p(X, Y, \tau)} \le \norm{C'}_{\bbB_q(Y', X^{\sun}, \tau)}.
    $
\end{theorem}

\begin{proof}
     Suppose that $C'\in \bbB_q(Y', X^{\sun}, (T^{\sun}(t))_{t\ge 0})$ and fix $\tau>0$. 
     Since $CT(\argument)x$ maps $X_1$ to $L^p([0,\tau], Y)$, so for each $x\in \dom(A)$, we can treat $CT(\argument)x$ as an element of $L^p([0,\tau], Y'')$.
     Therefore, we can employ the norming property of Banach space valued $L^p$-spaces \cite[Proposition~1.3.1]{HytoenenVanNeervenVeraarWeis2016} to obtain 
    \begin{align*}
        \norm{CT(\argument)x}_p &= \sup_{ \substack {y \in L^q([0,\tau], Y')\\ \norm{y}_q\le 1}} \modulus{ \int_0^{\tau}\duality{ CT(\tau-t)x}{ y(t)}\dx t}\\
                          & = \sup_{ \substack {y \in L^q([0,\tau], Y')\\ \norm{y}_q\le 1}} \modulus{ \duality{x}{\int_0^{\tau} T'(\tau-t)C'y(t)\dx t } }  \\
                          & = \sup_{ \substack {y \in L^q([0,\tau], Y')\\ \norm{y}_q\le 1}} \modulus{ \duality{x}{\int_0^{\tau} (T^{\sun})_{-1}(\tau-t)C'y(t)\dx t } }  \\
                          & \le \norm{C'}_{\bbB_q(Y',X^{\sun},\tau)}   \norm{x} 
    \end{align*}
    for all $x \in \dom(A)$.
    So, $C \in \bbC_p(X, Y, (T(t))_{t\ge 0})$ with the asserted inequality.
\end{proof}

\begin{remark}
    For $p=1$, Theorem~\ref{thm:weiss-sun-dual-observation-necessary} can be strengthened as follows:~if the control operator $C'\in \calL(Y',(X^{\sun})_{-1})$ is (zero-class) $\mathrm C$-admissible, then the observation operator $C\in \calL(X_1,Y)$ is (zero-class) $L^1$-admissible. Indeed, (zero-class) $\mathrm C$-admissibility of $C'$ implies (zero-class) $L^1$-admissibility of $C'' \in \calL((X^{\sun\sun})_1,Y'')$ by Theorem~\ref{thm:sun-dual-continuous}, and in turn, the claim.
\end{remark}

The forward implication in~\eqref{eq:admissibility-observation-sun-dual} is slightly subtle. While the case $p>1$ yields the desired implication, the case $p=1$ requires an additional assumption of zero-class admissibility, which emerges organically from our proof technique. Moreover, we are only able to show $\mathrm C$-admissibility of the control $C'$ in this case.

\begin{theorem}
    \label{thm:weiss-sun-dual-observation-sufficient-zero-class}
    Let $X$ and $Y$ be Banach spaces, let $(T(t))_{t\ge 0}$ be a $C_0$-semigroup on $X$,  
    let $C\in \calL(X_1,Y)$ be such that $C'(Y')\subseteq (X_1)^{\sun}= \left(X^{\sun}\right)_{-1}$, and let $p,q\in [1,\infty]$ be Hölder conjugates. 

    Assume that the observation operator $C$ is $L^p$-admissible. If $p>1$, then the control operator $C'\in \calL(Y',(X^{\sun})_{-1})$ is $L^q$-admissible. If $p=1$ and the admissibility of $C$ is zero-class, then the control operator $C'$ is 
    $\mathrm C$-admissible. In these cases,
        \[
            \norm{C'}_{\mathrm Z(Y',X^{\sun},\tau)} \le \norm{C}_{\bbC_p(X,Y,\tau)}
        \]
    with $\mathrm Z=\bbB_q$ and $\mathrm Z=\bbB_{\mathrm C}$ respectively.
\end{theorem}

\begin{proof}
    Let  $C \in \bbC_p(X, Y, (T(t))_{t\ge 0})$. For each $\tau>0$, the input operator $\Phi_{\tau}: L^q([0,\tau], Y')\to (X^{\sun})_{-1}$ given by
    \[
        \Phi_{\tau}: y \mapsto \int_0^{\tau} (T^{\sun})_{-1}(\tau-t)C'y(t)\dx t,
    \]
    is well-defined because $C'(Y')\subseteq \left(X^{\sun}\right)_{-1}$ and the (extrapolated) sun-dual semigroup is strongly continuous.
    Note from the proof of the first part \cite[Theorem~6.9(i)]{Weiss1989a} -- observing that the strong continuity of the dual semigroup was needed in this part of the argument merely for the integral in the definition of the input operator $\Phi_{\tau}$ to be well-defined -- that $\Ima \Phi_{\tau} \subseteq X'$ and
    \begin{equation}
        \label{eq:norm-of-dual-input}
        \norm{\Phi_{\tau}}_{\calL(L^q([0,\tau], Y'), X')}\le \norm{C}_{\bbC_p(X,Y,\tau)}
    \end{equation}

    \emph{Case 1}: Assume that $p>1$, fix an element $y \in L^q([0,\tau], Y')$, and extend $y$ by $0$ outside $[0,\tau]$. For $s\ge 0$, we write
    \[
        y_s:= y(\argument+s) \in L^q([0,\tau], Y').
    \]
    For each $0\le s<\tau$, we obtain
    \begin{align*}
        T'(s)\Phi_{\tau}y-\Phi_{\tau}y  
                                            & =  \int_0^{\tau} (T')_{-1}(\tau+s-t)C'y(t)\dx t - \Phi_{\tau}y  \\
                                            & = \int_0^{\tau+s} (T')_{-1}(\tau+s-t)C'y(t)\dx t - \Phi_{\tau}y \\
                                            & = \int_0^{s} (T')_{-1}(\tau+s-t)C'y(t)\dx t+\Phi_{\tau}y_s - \Phi_{\tau}y\\
                                            & = T'(\tau)\Phi_sy+\Phi_{\tau}(y_s-y).
    \end{align*}
    Applying triangle inequality and using the fact that $\Phi_s y \in X'$, we estimate
    \[
        \norm{T'(s)\Phi_{\tau}y-\Phi_{\tau}y}_{X'} \le \norm{T'(\tau)\Phi_s y }_{X'}+\norm{\Phi_{\tau}}_{\calL(L^q([0,\tau], Y'), X')} \norm{y_s-y}_{L^q([0,\tau], Y')}.
    \]
    Consequently,~\eqref{eq:norm-of-dual-input} along with $s<\tau$ yields
    \[
        \norm{T'(s)\Phi_{\tau}y-\Phi_{\tau}y}_{X'} \le \norm{C}_{\bbC_p(X,Y,\tau)}\left(
                                            \norm{T'(\tau)}\norm{y}_{L^{q}([0,s],Y')}  + \norm{y_s-y}_{L^q([0,\tau], Y')}\right).
    \]
    As $q<\infty$, both norms involving $y$  converge to $0$ as $s\downarrow 0$. 
    By definition of the sun-dual space, we infer that $\Phi_{\tau}y \in X^{\sun}$ and conclude the $L^q$-admissibility of $C'$.

    \emph{Case 2}: Let $C$ be zero-class $L^1$-admissible. Extend a fixed $y \in \mathrm C([0,\tau], Y')$ constantly outside $[0,\tau]$, so that 
    $
        y_s:= y(\argument+s) \in \mathrm C([0,\tau], Y').
    $
    In what follows, we consider the restriction of $\Phi_{\tau}$ to $\mathrm C([0,\tau], Y')$. For fixed $0\le s<\tau$, observe that
    \begin{align*}
          T'(s)\Phi_{\tau} y  - T'(\tau)\Phi_{s}y
                                            & =  \int_{0}^{\tau} (T')_{-1}(\tau+s-t)y(t)\dx t - \int_{0}^{s} (T')_{-1}(\tau+s-t)y(t)\dx t\\
                                            & = \int_s^{\tau} (T')_{-1}(\tau+s-t)y(t)\dx t
    \end{align*}
    which upon subtracting from
    \[
        \Phi_{\tau} y_s =   \int_0^{\tau} (T')_{-1}(\tau-t)y_s(t)\dx t = \int_s^{\tau+s} (T')_{-1}(\tau+s-t)y(t)\dx t
    \]
    gives us that
    \[
        \Phi_{\tau} y_s - T'(s)\Phi_{\tau} y  + T'(\tau)\Phi_{s}y  = \int_{\tau}^{\tau+s} (T')_{-1}(\tau+s-t)y(t)\dx t=   \Phi_sy_{\tau}.
    \]
    Therefore, in this case, we obtain that
    \[
        T'(s)\Phi_{\tau}y-\Phi_{\tau}y =T'(\tau)\Phi_sy- \Phi_s y_{\tau} +\Phi_{\tau}(y_s-y).
    \]
    As in Case 1, thanks to the triangle inequality, the zero-class admissibility and continuity of $y$, allow us to deduce that $\Phi_{\tau}y \in X^{\sun}$.
    
    Finally, the asserted inequality in the statement of the theorem holds by~\eqref{eq:norm-of-dual-input} and the closed graph theorem.
\end{proof}

We do not know whether the zero-class assumption in Theorem~\ref{thm:weiss-sun-dual-observation-sufficient-zero-class} can be dropped in the case of $p=1$, nor do we know if the assertion can be strengthened to $L^\infty$-admissibility of $C$.  While, $X^{\sun}=X'$ is sufficient for both \cite[Theorem~6.9(i)]{Weiss1989a}, another situation can be constructed for the case of positive systems. In light of the recent interest on infinite-dimensional positive systems \cite{Wintermayr2019, AroraGlueckPaunonenSchwenninger2024, Gantouh2022a, Gantouh2023}, we find it worthwhile to mention it in the following.

A non-empty subset $X_+$ of a Banach space $X$ is called a \emph{cone} if $\alpha X_+ +\beta X_+ \subseteq X_+$ for all $\alpha,\beta\ge 0$ and $X_+ \cap (-X_+)=\{0\}$. The cone $X_+$ induces a natural order on $X$ given by $x\le y$ if and only if $y-x\in X_+$. The Banach space $X$ together with a closed cone $X_+$ is called an \emph{ordered Banach space}.
We refer to \cite{AliprantisBurkinshaw2006, AliprantisTourky2007, AbramovichAliprantis2001} for the theory of ordered Banach spaces.

Closed subspaces of ordered Banach spaces are endowed with the induced order rendering them ordered Banach spaces as well.
The cone  $X_+$ is called \emph{generating} if $X=X_+ - X_+$ and it is called \emph{normal} if there exists $M\ge 1$ such that $0\le x\le y$ implies $\norm{x}\le M\norm{y}$. For example, $L^p(\Omega,\mu)$ with $p\in [1,\infty]$, $\mathrm C(K)$ for a compact set $K$, and $\mathrm C_0(L)$ for a locally compact set $L$ are ordered Banach spaces with the canonical cone being generating and normal.
The norm on $X$ is said to be \emph{additive on the positive cone} if
\[
    \norm{x+y}=\norm{x}+\norm{y}\quad \text{for all }x,y\in E_+.
\]
Both finite-dimensional spaces and $L^1(\Omega,\mu)$ fall into this category.  A non-empty set $C\subseteq X_+$ is called a \emph{face} of $X_+$ if for all $x,y,z\in X$, the inequality $0\le y\le x+z$ along with $x,z\in C$ implies $y\in C$.
Lastly, an operator $T$ between ordered Banach spaces $X$ and $Y$ is called \emph{positive} if $TX_+\subseteq Y_+$. The set of positive linear functionals on $X$ form a cone and turn $X'$ into an ordered Banach space. A $C_0$-semigroup on an ordered Banach space is called \emph{positive} if each semigroup operator is positive. In fact, the associated extrapolation space is also an ordered Banach space. For the definition and a  detailed analysis of the order on the extrapolation space, we refer to \cite[Section~2.2]{AroraGlueckPaunonenSchwenninger2024} and \cite[Section~4]{AroraGlueckSchwenningerPreprint}.

\begin{theorem}
    \label{thm:weiss-sun-dual-observation-sufficient-positive}
    Suppose that $X$ and $Y$ are ordered Banach spaces such that $X$ has a generating and normal cone. Let $(T(t))_{t\ge 0}$ be a positive $C_0$-semigroup on $X$ and $C\in \calL(X_1,Y)$ be such that $C'(Y')\subseteq (X_1)^{\sun}= \left(X^{\sun}\right)_{-1}$.
    
    If $C$ is positive, $(X^{\sun})_+$ is a face of $X'_+$, and the norm on $Y$ is additive on the positive cone, then
    \[
        C \in \bbC_p(X, Y, (T(t))_{t\ge 0}) \Rightarrow C'\in \bbB_q(Y', X^{\sun}, (T^{\sun}(t))_{t\ge 0})
    \]
    for Hölder conjugates $p,q\in [1,\infty]$.
\end{theorem}

An ordered Banach space is called a \emph{Banach lattice} if any two elements have a supremum and $\sup\{-x,x\} \le \sup\{-y,y\}$ implies $\norm{x}\le \norm{y}$.
If $X$ is a Banach lattice, then various sufficient conditions under which $(X^{\sun})_+$ is a face of $X'_+$ are given in \cite[Chapter~8]{vanNeerven1992}. In particular, this is the case when $X=\mathrm C(K)$.

\begin{proof}[Proof of Theorem~\ref{thm:weiss-sun-dual-observation-sufficient-positive}]
    For $\tau>0$, let $\Phi_{\tau}: L^q([0,\tau], Y')\to (X_1)^{\sun}$ be given by
    \[
        y \mapsto \int_0^{\tau} (T^{\sun})_{-1}(\tau-s)C'y(s)\dx s.
    \]
    As $C$ is positive, so it its dual $C'$. Together with positivity of the semigroup, this ensures positivity of $\Phi_{\tau}$.
    As explained in the proof of Theorem~\ref{thm:weiss-sun-dual-observation-sufficient-zero-class}, we know from \cite[Theorem~6.9(i)]{Weiss1989a} that $\Ima \Phi_{\tau} \subseteq X'$ and we're left to show that $\Ima\Phi_{\tau}\subseteq X^{\sun}$.

    Due to \cite[Proposition~1.4.2(3) and (2)]{BattyRobinson1984}, the assumption on $Y$ implies that $Y'$ has a unit, say $e$, i.e., $Y'=\bigcup_{\lambda>0} [-\lambda e,\lambda e]$.
    Denoting by $\mathbf e\in L^q([0,\tau], Y')$, the constant function taking value $e$, we have
     $\Phi_{\tau}\mathbf e  \in \dom( (A^{\sun})_{-1})= X^{\sun}$. As $(X^{\sun})_+$ is a face of $X'_+$ and $\Phi_{\tau}$ is positive, it follows that $\Phi_{\tau}y \in X^{\sun}$ for all $y \in L^q([0,\tau], Y')_+$. Finally, as $Y'$ has a unit, its cone -- and in turn the cone of $L^q([0,\tau],Y')$ -- is generating \cite[Lemma~2]{Ng1969}. It follows that $\Phi_{\tau}y \in X^{\sun}$ for all $y \in L^q([0,\tau], Y')$.
\end{proof}

\begin{remark}
    In Theorem~\ref{thm:weiss-sun-dual-observation-sufficient-positive}, one can argue as in \cite[Theorem~4.10]{AroraGlueckPaunonenSchwenninger2024} to weaken the assumption of positivity of $C$ to the condition: $C'$ maps the unit ball of $Y'$ into an order bounded subset of $(X_1)'$. Sufficient conditions for this property are available in \cite[Proposition~A.1]{AroraGlueckPaunonenSchwenninger2024}.
\end{remark}

\section{An Example}
    \label{sec:example}

Throughout this section, let $(R(t))_{t\ge 0}$ be the nilpotent right-shift semigroup on $X:=L^1\big([0,1]\big)$ with generator $A$. Consider the space of test functions
\[
    \widetilde \calD: = \{ \phi \in\mathrm C^{\infty}\big([0,1]\big): \phi(0)=0\}
\]
and for $g,f$ in the dual space $(\widetilde \calD)'$, write
\[
    g=\widetilde\partial f:\Leftrightarrow\duality{f}{\phi'}=-\duality{g}{\phi} \text{ for all }\phi\in \widetilde \calD.
\]
This allows for a convenient description of the extrapolation space corresponding to the dual:

\begin{proposition}
    The extrapolation space associated to the left shift semigroup on $X'$ is given by
    \begin{equation}
        \label{eq:extrapolation-left-translation}
        (X')_{-1}= \{ g \in (\widetilde \calD)': g= \widetilde\partial f \text{ for some } f\in L^\infty\big([0,1]\big)\}.
    \end{equation}
\end{proposition}

\begin{proof}
    We proceed as in \cite[Example~3.2.7]{Wintermayr2019}:~Let $j:X'\to (\widetilde \calD)'$ be the canonical embedding, i.e.,  $\duality{j(g)}{\phi}=\int_0^1 g(x)\phi(x)\dx x$ for all $\phi \in \widetilde \calD$. 
    For each $g\in X'$ and $\phi\in \widetilde \calD$, we have
    \[
        \duality{j(g)}{\phi} = \int_0^1((A')^{-1}g)'(x)\phi(x)\dx x
                             = -\int_0^1((A')^{-1}g)(x)\phi'(x)\dx x;
    \]
    here we have used that each $f\in \dom(A')$ satisfies $f(1)=0$ and $A'f=f'$. 
    Thus, $j$ can be extended to $j_{-1}: (X')_{-1}\to (\widetilde \calD)'$ as 
    \[
        \duality{j_{-1}(g)}{\phi} := -\duality{(A')_{-1}^{-1}g}{\phi'} \qquad \left(g \in (X')_{-1}, \phi \in \widetilde \calD\right).
    \]
    Next, let $g\in (X')_{-1}$ such that $j_{-1}(g)=0$ and set $h:=-(A')_{-1}^{-1}g$.  Then $\alpha:= \int_0^1 h(x)\dx x$ and $\phi(t):= \int_0^t (h(x)-\alpha)\dx x$ satisfy $\phi \in \widetilde \calD$ and $\int_0^1\alpha h(x)\dx x=\alpha^2$. Thus, 
    \[
        0  = \duality{j_{-1}(g)}{\phi}
           = \duality{h}{\phi'}
           = \int_0^1 h(x)(h(x)-\alpha)\dx x
           = \int_0^1 (h(x)-\alpha)^2\dx x,
    \]
    which shows that $h(x)=\alpha$ a.e. $x$. Next, choose $\phi \in \widetilde \calD$ with $\phi(1)\ne 0$. Then
    \[
        0  = \duality{j_{-1}(g)}{\phi}
           = \duality{h}{\phi'}
           = \alpha \int_0^1 \phi'(x)\dx x
           = \alpha \phi(1),
    \]
    from which it follows that $\alpha=0$. Injectivity of $(A')_{-1}^{-1}$ now implies that $g=0$. Summarising, $j_{-1}$ is actually an embedding and the following diagram
    \begin{center}
        \begin{tikzcd}
    	X'      \arrow[d, hook,"j" '] & & (X')_{-1} \arrow[ll, "(A')_{-1}^{-1}" '] \arrow[d, hook, "j_{-1}"] \\
            (\widetilde \calD)'  \arrow{rr}{\widetilde\partial}        & & (\widetilde \calD)'
        \end{tikzcd}
    \end{center}
    commutes. This proves that $(X')_{-1}$ is contained in the set on the right in~\eqref{eq:extrapolation-left-translation}. Conversely, let $g \in (\widetilde \calD)'$ such that  $g= \widetilde\partial f$  for some  $f\in X'$. Bijectivity of $(A')_{-1}$ implies that there exists $h\in (X')_{-1}$ such that $(A')_{-1}^{-1}h=f$. Commutativity of the above diagram implies that 
    $
        j_{-1}(h)=\widetilde\partial(j(f))
    $
    and so $h$ can be identified with $\widetilde \partial f$ which means $g=h\in (X')_{-1}$. This verifies~\eqref{eq:extrapolation-left-translation}.
    \end{proof}

Having the description of the dual extrapolation space at hand enables us to describe all (zero-class) $L^1$-admissible observation functionals associated to the right shift semigroup on $X$. In what follows, $\mathrm{BV}\big([0,1]\big)$ denotes the space of functions of bounded variation on $[0,1]$.

\begin{proposition}
    \label{prop:zero-class-l1-observation-right-translation}
    The complex-valued $L^1$-admissible observation operators are
    \begin{equation}
        \label{eq:l1-observation-right-translation}
        \bbC_1(X, \bbC, (R(t))_{t\ge 0})= \{C \in\calL(X_1,\bbC): C'(1)=\widetilde\partial c \text{ for some }c\in \mathrm{BV}([0,1])\}.
    \end{equation}
    Moreover, the admissibility of $C\in \bbC_1(X, \bbC, (R(t))_{t\ge 0})$ is zero-class if and only if $\widetilde\partial c$ has no atomic part.
\end{proposition}

\begin{proof}
    Let $C\in \calL(X_1,\bbC)$. Since $C'(\bbC)\subseteq (X_1)'=(X')_{-1}$, we obtain from~\eqref{eq:extrapolation-left-translation}, an element $c \in L^\infty\big([0,1]\big)$ such that $C'(1)=\widetilde\partial c$. Equivalently,
    \[
        Cf = \duality{C'(1)}{f} = \duality{\widetilde\partial c}{f}_{(X_1)' ,X_1}=\duality{\widetilde\partial c}{f}_{(\widetilde\calD)' ,\widetilde \calD} 
    \]
    for all $f \in \widetilde \calD$. In particular, for each $f \in \widetilde\calD$, we have
    \[
         C R(s)f =  \duality{\widetilde\partial c}{R(s)f}_{(\widetilde\calD)' ,\widetilde \calD}
                                        = \duality{\widetilde\partial c}{\bar f(s-\argument)}   _{(\widetilde\calD)' ,\widetilde \calD}
                                        = (\bar f \star\widetilde \partial c)(s);
    \]
    where $\bar f(x)=f(-x)$ and $f$ is extended by $0$ outside $[0,1]$.
    This means that 
    $C \in \bbC_1(X, \bbC, (R(t))_{t\ge 0})$ if and only if $\widetilde\partial c$ lies in $\calM\big([0,1]\big)$, the space of measures of bounded variation, see \cite[Theorem~2.5.8]{Grafakos2014}. Since $g=\widetilde\partial f$ implies that $g=\partial f$, so $\widetilde\partial c \in \calM\big([0,1]\big)$ is equivalent to $c \in \mathrm{BV}\big([0,1]\big)$ due to \cite[Proposition~3.6]{AmbrosioFuscoPallara2000}.
    The equality in~\eqref{eq:l1-observation-right-translation} is now immediate.
    
    Next, by the semigroup law, zero-class $L^1$-admissibility of $C$ is equivalent to $\lim_{\tau\downarrow 0}\int_{\xi}^{\xi+\tau} \modulus{C R(s)f}\dx s=0$ for each $\xi\in [0,1]$. 
    Therefore, the above computations show that $C$ is zero-class $L^1$-admissible if and only if there exists $c\in \mathrm{BV}\big([0,1]\big)$ such that $C'(1)=\widetilde\partial c$ and
    \begin{equation}
        \label{eq:zero-class-intermediate}
        \lim_{\tau\downarrow 0}\norm{f \star\widetilde \partial c}_{L^1([\xi,\xi+\tau], \bbC)}=0 \quad \text{for all }\xi\in [0,1]\text{ and } f\in \widetilde \calD.
    \end{equation}

    Now, let $c\in \mathrm{BV}\big([0,1]\big)$ and set $\mu:=\widetilde\partial c$. Using the Radon-Nikodym decomposition, we  write $\mu=\mu_a+\mu_j+\mu_c$ where $\mu_a$ is absolutely continuous with respect to the Lebesgue measure, $\mu_c$ is non-atomic part, and $\mu_j$ is the purely atomic part; see \cite[Section~3.2]{AmbrosioFuscoPallara2000}. If $\mu_j=0$, then a variation of Young's convolution inequality \cite[Page~54]{Folland2016} and mutual singularity of the measures, gives for each $f\in \widetilde\calD$:
    \begin{align*}
        \norm{f \star\widetilde \partial c}_{L^1([0,\tau], \bbC)} &\le \norm{\mu}_{\calM([0,\tau])} \norm{f}_{L^1([0,1])}\\
                                                                &= \big(\modulus{\mu_a}([0,\tau])+\modulus{\mu_c}([0,\tau])\big)\norm{f}_{L^1([0,1])},
    \end{align*}
    which converges to $0$ as $\tau\downarrow 0$. On the other hand, if $\mu_j\ne 0$, then there exists $\xi\in [0,1]$ such that $\modulus{\mu_j}(\{\xi\})\ne 0$. Therefore, $\modulus{\mu_j}([\xi,\xi+\tau])\not\to 0$ as $\tau\downarrow 0$. Once again, the mutual singularity of the measures yields
    \[
        \norm{\mu}_{\calM([\xi,\xi+\tau])}= \modulus{\mu_j}([\xi,\xi+\tau])+\modulus{\mu_a}([\xi,\xi+\tau])+\modulus{\mu_c}([\xi,\xi+\tau]).
    \]
    Noting as above that $\modulus{\mu_a}([\xi,\xi+\tau])+\modulus{\mu_c}([\xi,\xi+\tau]) \to 0$ as $\tau\downarrow 0$, it follows that $\lim_{\tau\downarrow 0}\norm{\mu}_{\calM([\xi,\xi+\tau])}\ne 0$. A suitable choice of $f\in \widetilde\calD$ thus implies that $\lim_{\tau\downarrow 0}\norm{f \star\widetilde \partial c}_{L^1([\xi,\xi+\tau], \bbC)}\ne 0$ and so the $L^1$-admissibility is not zero-class.
    An appeal to~\eqref{eq:zero-class-intermediate} now shows that the zero-class $L^1$-admissibility of $C$ is equivalent to the existence of $c\in \mathrm{BV}\big([0,1]\big)$ such that $C'(1)=\widetilde\partial c$ has no atomic part.
\end{proof}    

    Recall from \cite[Example~1.1(ii)]{vanNeerven1993} that the sun-dual semigroup associated to $(R(t))_{t\ge 0}$ is the nilpotent left translation semigroup $(L(t))_{t\ge 0}$ on
    \[
        X^{\sun} = \{ f \in\mathrm C\big([0,1]\big):f(1)=0\}.
    \]
    Moreover, we know from \cite[Example~5.1]{BatkaiJacobVoigtWintermayr2018} that
    \[
        (X^{\sun})_{-1}= \{ g \in \calD': g= \partial f \text{ for some } f\in X^{\sun}\};
    \]
    where $\calD$ is the usual space of test functions on $[0,1]$, i.e.,
    $
        \calD: = \mathrm C_{\mathrm c}^\infty\big((0,1)\big).
    $
    While the direct computation of admissible control operators associated to $(L(t))_{t\ge 0}$ on $X^{\sun}$ is tedious, our results in the prequel along with the analysis in the present section allow us to characterise all $\mathrm C$-admissible rank-one control operators:
    \begin{proposition}
        \label{prop:admissible-control-left-translation}
        The set $\bbB_{\mathrm C}(\bbC, X^{\sun}, (L(t))_{t\ge 0})$ can be described as
        \[
            \{ C':  C \in \calL(X_1,\bbC)  \text{ and }C'(1)=\partial b=\widetilde\partial c \text{ for some }b\in X^{\sun}, c\in \mathrm{BV}\big([0,1]\big)\}. 
        \]
    \end{proposition}

    \begin{proof}
        Firstly, let $C\in\calL(X_1,Y)$ be such that $C'(1)=\partial b=\widetilde\partial c$ for some $b\in X^{\sun}$ and $c\in \mathrm{BV}\big([0,1]\big)$. By the description of $(X^{\sun})_{-1}$, we get $C'(\bbC)\subseteq (X^{\sun})_{-1}$. 
        In fact, the continuity of $b$, in particular,  also means that $\widetilde\partial c =\partial b$ has no atomic part. Together, Proposition~\ref{prop:zero-class-l1-observation-right-translation} and Theorem~\ref{thm:weiss-sun-dual-observation-sufficient-zero-class} now imply that $C'\in\bbB_{\mathrm C}(\bbC, X^{\sun}, (L(t))_{t\ge 0})$.

        Conversely, let $B\in \bbB_{\mathrm C}(\bbC, X^{\sun}, (L(t))_{t\ge 0})$, then sun-reflexivity of $X$ -- which is known from \cite[Example~1.3(ii)]{vanNeerven1993} -- along with Theorem~\ref{thm:sun-dual-continuous} yields that $B'\in \bbC_1(X, \bbC, (R(t))_{t\ge 0})$. Thus there exists $c\in \mathrm{BV}\big([0,1]\big)$ such that $B''(1)=\widetilde\partial c$. Also, as $B''(1)=B(1) \in (X^{\sun})_{-1}$, there exists $b\in X^{\sun}$ such that $B''(1)=\partial b$. Thus, $B=B''$ has the desired form.
    \end{proof}

    \begin{remarks}
        (a) The proof of Proposition~\ref{prop:admissible-control-left-translation} even shows that the admissibility of each element of $\bbB_{\mathrm C}(\bbC, X^{\sun}, (L(t))_{t\ge 0})$ is zero-class; cf. Theorem~\ref{thm:weiss-sun-dual-observation-sufficient-zero-class}.
        
        (b) It can be inferred from \cite[Corollary~4.7]{AroraGlueckPaunonenSchwenninger2024} that every positive $B\in \calL(\bbC, (X^{\sun})_{-1})$ is zero-class $\mathrm C$-admissible.

        (c) While $g=\widetilde f$ implies $g=\partial f$, the converse is not true in general. Therefore, we do not know whether in the description of $\bbB_{\mathrm C}(\bbC, X^{\sun}, (L(t))_{t\ge 0})$ from Proposition \ref{prop:admissible-control-left-translation} we can simply write $C'(1)=\partial b$  for some $b\in X^{\sun}\cap \mathrm{BV}\big([0,1]\big)$.
    \end{remarks}

\section{Concluding remarks: Generalizations and further directions}
    \label{sec:conclusion}

\subsection{Beyond the sun dual space}

A closer look at the proofs in Sections~\ref{sec:control} and~\ref{sec:observation} suggest that the results can be generalized  for closed subspaces of the sun-dual space that leave the dual semigroup invariant as long as the norming property holds.

Let $(T(t))_{t\ge 0}$ be a $C_0$-semigroup on a Banach space $X$ with generator $A$. Let $X^{\dotbox}$ be a closed subspace of $X^\sun$, which is invariant under the dual semigroup $(T'(t))_{t\ge 0}$ and norming for $X$, that is, 
\begin{equation*}
    \norm{x}' :=\sup\left\{\modulus{\duality{x^{\dotbox}}{x}_{X' , X}}\colon x^{\dotbox}\in X^{\dotbox}\right\}, \quad x\in X,
\end{equation*}
defines an equivalent norm on $X$. Then, the restriction $\left(T^{\dotbox}(t)\right)_{t\ge 0}$ of the dual semigroup on $X^{\dotbox}$ generates a strongly continuous semigroup on $X^{\dotbox}$ whose generator is given by the part of $A'$ in $X^{\dotbox}$, which follows completely analogous as in the case of $X^{\dotbox}=X^{\sun}$, see e.g., \cite[Proposition~II.2.6]{EngelNagel2000}. With obvious adaptations of the proofs of Theorems~\ref{thm:sun-dual-continuous}, \ref{thm:weiss-sun-dual-control}, \ref{thm:weiss-sun-dual-observation-necessary}, and~\ref{thm:weiss-sun-dual-observation-sufficient-positive}, we obtain the following:

\begin{theorem}
    Let $(T(t))_{t\ge 0}$ be a $C_0$-semigroup on a Banach space $X$ and $X^{\dotbox}$ be a closed subspace of $X^{\sun}$ which is assumed to be norming for $X$ and $T(t)'$-invariant for all $t>0$. Further, let $p,q\in [1,\infty]$ be Hölder conjugates.
    \begin{enumerate}
        \item Let $U$ be a Banach space and let $B\in \calL(U, X_{-1})$ be a control operator.
        \begin{enumerate}
            \item The operator $B$ is (zero-class) $\mathrm C$-admissible if and only if the observation operator $B' \in \calL\left(\left(X^{\dotbox}\right)_1, U'\right)$ is (zero-class) $L^1$-admissible.

            \item If $B \in \bbB_p(U, X, (T(t))_{t\ge 0}))$, then $B' \in \bbC_q\left(X^{\dotbox},U',\left(T^{\dotbox}(t)\right)_{t\ge 0}\right)$ with 
            \[
                \norm{B'}_{\bbC_q\left(X^{\dotbox},U',\tau\right)} \le \norm{B}_{\bbB_p(U, X, \tau)}.
            \]
            The converse is true for $p<\infty$.
        \end{enumerate}

        \item Let $Y$ be a Banach space and $C\in \calL(X_1,Y)$ such that $C'(Y')\subseteq \left(X^{\dotbox}\right)_{-1}$.
        \begin{enumerate}
            \item If $C'\in \bbB_q\left(Y', X^{\dotbox}, \left(T^{\dotbox}(t)\right)_{t\ge 0}\right)$, then $C \in \bbC_p(X, Y, (T(t))_{t\ge 0})$ with
            $
                \norm{C}_{\bbC_p(X, Y, \tau)} \le \norm{C'}_{\bbB_q\left(Y', X^{\dotbox}, \tau\right)}.
            $

            \item Suppose that $X$ and $Y$ are even ordered Banach spaces such that $X$ has a generating and normal cone, $\left(X^{\dotbox}\right)_+$ is a face of $X'_+$, and the norm on $Y$ is additive on the positive cone.
            
            In addition, let the semigroup $(T(t))_{t\ge 0}$ and the operator $C$ be positive. Then
            $
                C \in \bbC_p(X, Y, (T(t))_{t\ge 0}) \Rightarrow C'\in \bbB_q\left(Y', X^{\dotbox}, (T^{\dotbox}(t))_{t\ge 0}\right).
            $
        \end{enumerate}
    \end{enumerate}
\end{theorem}

We don't know whether Theorem~\ref{thm:weiss-sun-dual-observation-sufficient-zero-class} can be adapted to this setting.

The above generalization of the sun-dual approach is not that artificial as it may seem on first glance. For sectorial operators, a concrete case of such a subspace of the sun-dual is the so-called \emph{moon-dual} space  defined as
\[
    X^{\moon} := \overline{\dom(A') }\cap \overline{\Ima (A')  };
\]
where the closures are taken in $X'$. Recalling from \cite[Theorem~1.3.1]{vanNeerven1992} that $X^{\sun} = \overline{\dom(A') }$, we immediately see that $X^{\moon}$ is a closed subspace of $X^{\sun}$ that is invariant under the dual semigroup $(T(t)')_{t\ge 0}$. It can also be shown that if $A$ is a sectorial operator, then $X^{\moon}$ is norming for $X$, see e.g., \cite[Appendix~15.B]{KunstmannWeis2004}. 
The moon dual space was first introduced by Fröhlich in his PhD thesis \cite[Section~6.2]{Fröhlich2003} as a tool in characterizing bounded $H^{\infty}$-calculus; ; see also \cite[Section~3]{FröhlichWeis2006}, and \cite{KaltonLoristWeis2023, Kriegler2009, KunstmannWeis2004}.

\subsection{Controllability and observability}

The contribution of the results derived in Sections~\ref{sec:control} and~\ref{sec:observation} can be interpreted as ``replacing reflexivity of the involved state space $X$ by considering the sun-dual (semigroup) instead of the dual (semigroup)". It seems natural to approach the related -- and formally dual -- concepts of controllable and observable linear systems similarly. Let $X, U$, and $X_F$ be Banach spaces and let $\mathrm Z$ be a placeholder for $\mathrm C$ (to denote continuous functions) or $L^p$ with $p\in [1,\infty]$. Recall that the tuple $(A,B,F,\tau)$ -- where $A$ generates a strongly continuous semigroup $(T(t))_{t\ge 0}$ on $X$,  $B\in\mathcal{L}(U,X_{-1})$, and $F\in \mathcal{L}(X_{F},X)$ -- is called \emph{$\mathrm{Z}$-controllable} if $B$ is $\mathrm{Z}$-admissible and $F(X_{F})\subseteq \Phi_{\tau}\big(\mathrm Z([0,\tau],U)\big)$. In the special case $X_{F}=X$ and $F=I_{X}$, the concept reduces to \emph{exact controllability}, whereas the case $X_{F}=X$ and $F=T(\tau)$ refer to the notion of \emph{null controllability}, see, e.g.,~\cite{Carja1988}. On the other hand, for $C\in\mathcal{L}(X_{1},Y)$ and $G\in\mathcal{L}(X,Y_{G})$ for some Banach spaces $Y$ and $Y_{G}$, we call $(C,A,G,\tau)$ \emph{$\mathrm{Z}$-observable} if $C$ is $\mathrm{Z}$-admissible and 
\[
    \norm{Gx}\leq \norm{\Psi_{\tau}x}_{\mathrm Z([0,\tau],Y)}\quad \forall~x\in X.
\]
For $G=I_{X}$ and $G=T(\tau)$, this notion corresponds to \emph{exact observability} and \emph{final state observability} (with respect to $\mathrm Z$), respectively. 

It is well-known that controllability and observability are dual, i.e., $(A,B,F,\tau)$ is controllable if and only if $(B',A{'},F{'},\tau)$ is observable provided that $X$ is a reflexive space or if $F=I_{X}$. In the latter case, this is a rather direct consequence of the open mapping theorem. Equivalences of this duality are often argued by abstract arguments originating from Douglas' work \cite{Douglas1966}; cf., \cite{Carja1988} and the references therein. 

It natural to pose the question if such duality results can also be derived in the context of sun-duality presented in the earlier sections of this work. Clearly, as for the admissibility involved in the definition of controllable and observable this follows as an application of the derived results. Yet, it seems far more involved to derive equivalences for the notions in the sun-dual context, already in the most prominent case of $F=I$ (or $G=I)$. Let us briefly indicate these obstructions: Suppose that  $\left(B{'}\restrict{(X^{\sun})_{1}},A^{\sun}, I_{X^{\sun}},\tau\right)$ is $L^{1}$-observable, which means that the mapping 
\begin{equation*}
    \Psi_{\tau}: X^{\sun}\to L^{1}([0,\tau],U'),x^{\sun}\mapsto B'T^{\sun}(\cdot)x^{\sun}
\end{equation*}
is well-defined as bounded operator and bounded from below. In order to show that $(A,B,I_{X},\tau)$ is controllable, by , we have to show that mapping
\[
    \Phi_{\tau}:\mathrm C([0,\tau], U)\to X, u\mapsto \int_{0}^{t}T_{-1}(t-s)Bu(s)\dx s
\]
which is well-defined due to Theorem \ref{thm:sun-dual-continuous}, is surjective. The standard argument -- which would work if $\Psi_{\tau}$ and $\Phi_{\tau}$ were dual operators -- is to show surjectivity relying on a Hahn-Banach argument. In the more general setting here, where $X^{\sun}$ is in general not equal to the dual $X'$, this seems to be out of reach. The converse direction, that $(A,B,F,\tau)$ $C$-controllable implies that $\left(B'\restrict{(X^{\sun})_{1}},A^{\sun}, I_{X^{\sun}},\tau\right)$ is $L^{1}$-observable is rather easy to show, just as in the classical case. 
The difficulties seem to become even more obstructive if $F$ is not the identity operator. While the presented results of this article seem to clarify the first step to handle these questions, a rigorous treatment of this duality for $X^{\sun}\neq X'$ and non-reflexive spaces $\mathrm Z([0,\tau];U)$ seems to be an interesting path for future research.

\subsection*{Acknowledgements} 
The first author was funded by the Deutsche Forschungsgemeinschaft (DFG, German Research Foundation) -- 523942381. The authors are indebted to the anonymous reviewers for their very useful comments and especially for bringing up the relevance of the presented results in the context of the generalizations discussed in Section 6.
    
\bibliographystyle{plainurl}
\bibliography{literature}

@inproceedings{Schwenninger2020,
   author = {Schwenninger, Felix L.},
   title = {Input-to-state stability for parabolic boundary control:linear and semilinear systems},
   series = {Control Theory of Infinite-Dimensional Systems},
   year = {2020},
    booktitle = {Control Theory of Infinite-Dimensional Systems},
   publisher = {Birkhäuser},
   pages = {83-116},
   ISBN = {978-3-030-35898-3},
   DOI = {https://doi.org/10.1007/978-3-030-35898-3_4},
   type = {Conference Proceedings}
}

@article{lotz1985,
	author = {Lotz, Heinrich P.},
	doi = {10.1007/BF01160459},
	journal = {Math. Z.},
	number = {2},
	pages = {207--220},
	title = {Uniform convergence of operators on {$L^\infty$} and similar spaces},
	volume = {190},
	year = {1985},
}

@article{Spe2020,
	author = {Spek, Len and Kuznetsov, Yuri A. and van Gils, Stephan A.},
	doi = {10.1186/s13408-020-00098-5},
	journal = {J. Math. Neurosci.},
	number = {1},
	pages = {21},
	title = {Neural field models with transmission delays and diffusion},
	volume = {10},
	year = {2020},
	doi = {10.1186/s13408-020-00098-5}
}

@Book{diekmannvangils1995,
 Author = {Diekmann, Odo and van Gils, Stephan A. and Verduyn Lunel, Sjoerd M. and Walther, Hans-Otto},
 Title = {Delay equations. {Functional}-, complex-, and nonlinear analysis},
 Series = {Applied Mathematical Sciences},
 ISSN = {0066-5452},
 Volume = {110},
 ISBN = {0-387-94416-8},
 Year = {1995},
 Publisher = {New York, NY: Springer-Verlag},
 Language = {English},
 zbMATH = {770062},
 Zbl = {0826.34002},
doi={10.1007/978-1-4612-4206-2}
}

@Article{controlsundual2,
 Author = {Heijmans, Henk J. A. M.},
 Title = {Semigroup theory for control on sun-reflexive {Banach} spaces},
 Journal = {IMA Journal of Mathematical Control and Information},
 ISSN = {0265-0754},
 Volume = {4},
 Pages = {111--129},
 Year = {1987},
 Language = {English},
 DOI = {10.1093/imamci/4.2.111},
 zbMATH = {4020950},
 Zbl = {0627.93031}
}

@incollection{controlsundual1,
	author = {Desch, Wolfgang and Schappacher, Wilhelm and Fa{\v{s}}angov{\'a}, Eva and Milota, Jaroslav},
	booktitle = {Evolution equations and their applications in physical and life sciences. Proceeding of the Bad Herrenalb (Karlsruhe) conference, Germany, 1999},
	isbn = {0-8247-9010-3},
	language = {English},
	pages = {247--254},
	publisher = {New York, NY: Marcel Dekker},
	title = {Infinite horizon {Riccati} operators in nonreflexive spaces},
	year = {2001},
	zbl = {0985.47033},
	zbmath = {1588430},
        doi={10.1201/9780429187810}
}

@inproceedings{diekmanvgils1991,
	address = {Berlin, Heidelberg},
	author = {Diekmann, Odo and van Gils, Stephan A.},
	booktitle = {Singularity Theory and its Applications},
	editor = {Roberts, Mark and Stewart, Ian},
	isbn = {978-3-540-47047-2},
	pages = {122--141},
	publisher = {Springer Berlin Heidelberg},
	title = {The center manifold for delay equations in the light of suns and stars},
        doi={10.1007/BFb0085429},
	year = {1991}}

@Article{BattyRobinson1984,
 Author = {{Batty}, Charles J. K. and {Robinson}, Derek W.},
 Title = {Positive one-parameter semigroups on ordered {Banach} spaces},
 Journal = {Acta Applicandae Mathematicae},
 ISSN = {0167-8019},
 Volume = {2},
 Pages = {221--296},
 Year = {1984},
 Language = {English},
 DOI = {10.1007/BF02280855},
 zbMATH = {3883065},
 Zbl = {0554.47022}
}

@Article{HaakLeMerdy2005,
 Author = {Haak, Bernhard and Le Merdy, Christian},
 Title = {{{\(\alpha\)}}-admissibility of observation and control operators},
 Journal = {Houston Journal of Mathematics},
 ISSN = {0362-1588},
 Volume = {31},
 Number = {4},
 Pages = {1153--1167},
 Year = {2005},
 Language = {English},
 zbMATH = {5020259},
 Zbl = {1099.47036},
url={https://www.math.u-bordeaux.fr/~bhaak/recherche/haaklemerdy.pdf}
}

@PhdThesis{Haak2004,
  author = {Haak, Bernhard H.},
  title  = {{Kontrolltheorie in Banachräumen und quadratische Abschätzungen}},
  school = {Universität Karlsruhe},
  year   = {2004},
  doi   = {10.5445/KSP/1000001171},
  language  = {German},
}

@Article{JacobPartington2001,
 Author = {Jacob, Birgit and Partington, Jonathan R.},
 Title = {The {Weiss} conjecture on admissibility of observation operators for contraction semigroups},
 Journal = {Integral Equations and Operator Theory},
 ISSN = {0378-620X},
 Volume = {40},
 Number = {2},
 Pages = {231--243},
 Year = {2001},
 Language = {English},
 DOI = {10.1007/BF01301467},
 zbMATH = {1642555},
 Zbl = {1031.93107}
}

@InCollection{JacobPartington2004,
 Author = {Jacob, Birgit and Partington, Jonathan R.},
 Title = {Admissibility of control and observation operators for semigroups: a survey},
 BookTitle = {Current trends in operator theory and its applications. Proceedings of IWOTA 2002},
 ISBN = {3-7643-7067-X},
 Pages = {199--221},
 Year = {2004},
 Publisher = {Basel: Birkh{\"a}user},
 Language = {English},
 zbMATH = {2151324},
 Zbl = {1083.93025},
doi={10.1007/978-3-0348-7881-4_10}
}

@Article{JacobPartingtonPott2014,
 Author = {Jacob, Birgit and Partington, Jonathan R. and Pott, Sandra},
 Title = {Applications of {Laplace}-{Carleson} embeddings to admissibility and controllability},
 Journal = {SIAM Journal on Control and Optimization},
 ISSN = {0363-0129},
 Volume = {52},
 Number = {2},
 Pages = {1299--1313},
 Year = {2014},
 Language = {English},
 DOI = {10.1137/120894750},
 zbMATH = {6323058},
 Zbl = {1294.30098}
}

@Article{JacobSchwenningerWintermayr2022,
 Author = {Jacob, Birgit and Schwenninger, Felix L. and Wintermayr, Jens},
 Title = {A refinement of {Baillon}'s theorem on maximal regularity},
 Journal = {Studia Mathematica},
 ISSN = {0039-3223},
 Volume = {263},
 Number = {2},
 Pages = {141--158},
 Year = {2022},
 Language = {English},
 DOI = {10.4064/sm200731-20-3},
 zbMATH = {7500265}
}

@Article{JacobSchwenningerZwart2019,
 Author = {Jacob, Birgit and Schwenninger, Felix L. and Zwart, Hans},
 Title = {On continuity of solutions for parabolic control systems and input-to-state stability},
 Journal = {Journal of Differential Equations},
 ISSN = {0022-0396},
 Volume = {266},
 Number = {10},
 Pages = {6284--6306},
 Year = {2019},
 Language = {English},
 DOI = {10.1016/j.jde.2018.11.004},
 zbMATH = {7036305}
}

@Article{MironchenkoPrieur2020,
 Author = {Mironchenko, Andrii and Prieur, Christophe},
 Title = {Input-to-state stability of infinite-dimensional systems: recent results and open questions},
 Journal = {SIAM Review},
 ISSN = {0036-1445},
 Volume = {62},
 Number = {3},
 Pages = {529--614},
 Year = {2020},
 Language = {English},
 DOI = {10.1137/19M1291248},
 zbMATH = {7279892},
 Zbl = {1453.93207}
}

@Book{Salamon1984,
 Author = {Salamon, Dietmar},
 Title = {Control and observation of neutral systems},
 Series = {Research Notes in Mathematics (San Francisco)},
 Volume = {91},
 Year = {1984},
 Publisher = {Pitman Publishing, London},
 Language = {English},
 zbMATH = {3871134},
 Zbl = {0546.93041},
url={https://people.math.ethz.ch/~salamon/PREPRINTS/ConObsNeuSys.pdf}
}

@Article{Salamon1987,
 Author = {Salamon, Dietmar},
 Title = {Infinite dimensional linear systems with unbounded control and observation: {A} functional analytic approach},
 Journal = {Transactions of the American Mathematical Society},
 ISSN = {0002-9947},
 Volume = {300},
 Pages = {383--431},
 Year = {1987},
 Language = {English},
 DOI = {10.2307/2000351},
 zbMATH = {4012437},
 Zbl = {0623.93040}
}

@Book{Staffans2005,
 Author = {Staffans, Olof Johan},
 Title = {Well-posed linear systems},
 Series = {Encyclopedia of Mathematics and Its Applications},
 ISSN = {0953-4806},
 Volume = {103},
 ISBN = {0-521-82584-9},
 Year = {2005},
 Publisher = {Cambridge: Cambridge University Press},
 Language = {English},
 zbMATH = {2136427},
 Zbl = {1057.93001},
doi={10.1017/CBO9780511543197}
}

@Book{TucsnakWeiss2009,
 Author = {Tucsnak, Marius and Weiss, George},
 Title = {Observation and control for operator semigroups},
 Series = {Birkh{\"a}user Advanced Texts. Basler Lehrb{\"u}cher},
 ISSN = {1019-6242},
 ISBN = {978-3-7643-8993-2; 978-3-7643-8994-9},
 Year = {2009},
 Publisher = {Basel: Birkh{\"a}user},
 Language = {English},
 zbMATH = {5359001},
 Zbl = {1188.93002},
doi={10.1007/978-3-7643-8994-9}
}

@article {TucsnakWeiss2014,
    AUTHOR = {Tucsnak, Marius and Weiss, George},
     TITLE = {Well-posed systems -- the {LTI} case and beyond},
   JOURNAL = {Automatica J. IFAC},
  FJOURNAL = {Automatica. A Journal of IFAC, the International Federation of
              Automatic Control},
    VOLUME = {50},
      YEAR = {2014},
    NUMBER = {7},
     PAGES = {1757--1779},
      ISSN = {0005-1098},
   MRCLASS = {93C05 (93-02 93C25)},
  MRNUMBER = {3230878},
       DOI = {10.1016/j.automatica.2014.04.016},
}

@Article{Weiss1989a,
 Author = {Weiss, George},
 Title = {Admissible observation operators for linear semigroups},
 Journal = {Israel Journal of Mathematics},
 ISSN = {0021-2172},
 Volume = {65},
 Number = {1},
 Pages = {17--43},
 Year = {1989},
 Language = {English},
 DOI = {10.1007/BF02788172},
 zbMATH = {4140600},
 Zbl = {0696.47040}
}

@Article{Weiss1989b,
 Author = {Weiss, George},
 Title = {Admissibility of unbounded control operators},
 Journal = {SIAM Journal on Control and Optimization},
 ISSN = {0363-0129},
 Volume = {27},
 Number = {3},
 Pages = {527--545},
 Year = {1989},
 Language = {English},
 DOI = {10.1137/0327028},
 zbMATH = {4123628},
 Zbl = {0685.93043}
}

@PhdThesis{Wintermayr2019,
  author = {Jens {Wintermayr}},
  title  = {{Positivity in perturbation theory and infinite-dimensional systems}},
  school = {Bergische Universität Wuppertal},
  year   = {2019},
  doi   = {10.25926/pd7n-9570},
  language  = {English},
}

@book {AliprantisTourky2007,
    AUTHOR = {Aliprantis, Charalambos D. and Tourky, Rabee},
     TITLE = {Cones and duality},
    SERIES = {Graduate Studies in Mathematics},
    VOLUME = {84},
 PUBLISHER = {American Mathematical Society, Providence, RI},
      YEAR = {2007},
     PAGES = {xiv+279},
      ISBN = {978-0-8218-4146-4},
   MRCLASS = {46A40 (46B40 46N10 47B60)},
  MRNUMBER = {2317344},
MRREVIEWER = {Nicolae Popovici},
       DOI = {10.1090/gsm/084},
}

@Book{EngelNagel2000,
  title     = {{One-parameter semigroups for linear evolution equations}},
  publisher = {Berlin: Springer},
  year      = {2000},
  author    = {Klaus-Jochen {Engel} and Rainer {Nagel}},
  volume    = {194},
  isbn      = {0-387-98463-1/hbk},
  fjournal  = {{Graduate Texts in Mathematics}},
  issn      = {0072-5285},
  journal   = {{Grad. Texts Math.}},
  language  = {English},
  msc2010   = {47D06 47-02 34K30 34K35 35K30},
  pages     = {xxi + 586},
  zbl       = {0952.47036},
   DOI = {10.1007/b97696},
}

@article{Gantouh2022a,
  eprint = {2208.10617v3},
  author = {El Gantouh, Yassine},
  title = {Positivity of infinite-dimensional linear systems},
  year = {2022},
archivePrefix={arXiv},
  note={Preprint}
}

@Article{Gantouh2023,
 author = {El Gantouh, Yassine},
 Title = {Boundary approximate controllability under positivity constraints of infinite-dimensional control systems},
 Journal = {Journal of Optimization Theory and Applications},
 ISSN = {0022-3239},
 Volume = {198},
 Number = {2},
 Pages = {449--478},
 Year = {2023},
 Language = {English},
 DOI = {10.1007/s10957-023-02200-9},
 zbMATH = {7740095},
 Zbl = {1522.93033}
}

@Book{vanNeerven1992,
 Author = {van Neerven, Jan},
 Title = {The adjoint of a semigroup of linear operators},
 Series = {Lecture Notes in Mathematics},
 ISSN = {0075-8434},
 Volume = {1529},
 ISBN = {3-540-56260-5},
 Year = {1992},
 Publisher = {Berlin: Springer-Verlag},
 Language = {English},
 zbMATH = {108530},
 Zbl = {0780.47026},
doi={10.1007/BFb0085008}
}

@Book{AliprantisBurkinshaw2006,
 Author = {Aliprantis, Charalambos D. and Burkinshaw, Owen},
 Title = {Positive operators},
 Edition = {Reprint of the 1985 original},
 ISBN = {1-4020-5007-0},
 Year = {2006},
 Publisher = {Berlin: Springer},
 Language = {English},
 zbMATH = {5061163},
 Zbl = {1098.47001},
doi={10.1007/978-1-4020-5008-4}
}

@Article{Travis1981,
 Author = {Travis, Curtis C.},
 Title = {Differentiability of weak solutions to an abstract inhomogeneous differential equation},
 Journal = {Proceedings of the American Mathematical Society},
 ISSN = {0002-9939},
 Volume = {82},
 Pages = {425--430},
 Year = {1981},
 Language = {English},
 DOI = {10.2307/2043955},
 zbMATH = {3760585},
 Zbl = {0484.34044}
}

@Article{BatkaiJacobVoigtWintermayr2018,
 Author = {B{\'a}tkai, Andr{\'a}s and Jacob, Birgit and Voigt, J{\"u}rgen and Wintermayr, Jens},
 Title = {Perturbations of positive semigroups on {AM}-spaces},
 Journal = {Semigroup Forum},
 ISSN = {0037-1912},
 Volume = {96},
 Number = {2},
 Pages = {333--347},
 Year = {2018},
 Language = {English},
 DOI = {10.1007/s00233-017-9879-0},
 zbMATH = {6893038}
}

@Article{AroraGlueckPaunonenSchwenninger2024,
    author = {Sahiba {Arora} and Jochen {Gl\"uck} and Lassi {Paunonen} and Felix L. {Schwenninger}},
 fjournal = {Journal of Differential Equations},
 journal = {J. Differ. Equations},
 issn = {0022-0396},
 volume = {440},
 pages = {35},
 note = {Id/No 113435},
 year = {2025},
 language = {English},
 doi = {10.1016/j.jde.2025.113435},
 zbMATH = {8056303}
}

@article{vanNeerven1993,
 Author = {van Neerven, Jan},
 Title = {Some recent results on adjoint semigroups},
 Journal = {CWI Quarterly},
 ISSN = {0922-5366},
 Volume = {6},
 Number = {2},
 Pages = {139--153},
 Year = {1993},
 Language = {English},
 zbMATH = {431357},
 Zbl = {0807.47029},
url={https://typeset.io/pdf/some-recent-results-on-adjoint-semigroups-fl7x0d2scz.pdf}
}

@Book{AmbrosioFuscoPallara2000,
 Author = {Ambrosio, Luigi and Fusco, Nicola and Pallara, Diego},
 Title = {Functions of bounded variation and free discontinuity problems},
 Series = {Oxford Mathematical Monographs},
 ISBN = {0-19-850245-1},
 Year = {2000},
 Publisher = {Oxford: Clarendon Press},
 Language = {English},
 zbMATH = {1448982},
 Zbl = {0957.49001},
doi={10.1093/oso/9780198502456.001.0001}
}

@Book{Folland2016,
 Author = {Folland, Gerald B.},
 Title = {A course in abstract harmonic analysis},
 Edition = {2nd updated},
 Series = {Textbooks in Mathematics},
 ISBN = {978-1-4987-2713-6; 978-1-4987-2715-0},
 Year = {2016},
 Publisher = {Boca Raton, FL: CRC Press},
 Language = {English},
 DOI = {10.1201/b19172},
 zbMATH = {6473590},
 Zbl = {1342.43001}
}

@Article{EmirsjlowTownley2000,
 Author = {Emirsjlow, Zbigniew and Townley, Stuart},
 Title = {From {PDEs} with boundary control to the abstract state equation with an unbounded input operator: a tutorial},
 Journal = {European Journal of Control},
 ISSN = {0947-3580},
 Volume = {6},
 Number = {1},
 Pages = {27--53},
 Year = {2000},
 Language = {English},
 DOI = {10.1016/S0947-3580(00)70908-3},
 zbMATH = {5597946},
 Zbl = {1167.93356}
}

@Article{JacobNabiullinPartingtonSchwenninger2018,
 Author = {Jacob, Birgit and Nabiullin, Robert and Partington, Jonathan R. and Schwenninger, Felix L.},
 Title = {Infinite-dimensional input-to-state stability and {Orlicz} spaces},
 Journal = {SIAM Journal on Control and Optimization},
 ISSN = {0363-0129},
 Volume = {56},
 Number = {2},
 Pages = {868--889},
 Year = {2018},
 Language = {English},
 DOI = {10.1137/16M1099467},
 zbMATH = {6850894},
 Zbl = {1390.93661}
}

@Inproceedings{PreusslerSchwenninger2024,
 Author = {Preu{\ss}ler, Philip and Schwenninger, Felix L.},
 title = {On checking {{\(L^p\)}}-admissibility for parabolic control systems},
 booktitle = {Systems theory and PDEs. Open problems, recent results, and new directions. Based on the first workshop on systems theory and PDEs, WOSTAP, Freiberg, Germany, July 2022},
 isbn = {978-3-031-64990-5; 978-3-031-64993-6; 978-3-031-64991-2},
 pages = {219--256},
 year = {2024},
 publisher = {Cham: Birkh{\"a}user},
 language = {English},
 doi = {10.1007/978-3-031-64991-2_9},
 zbMATH = {7949846},
 Zbl = {1553.35136}
}

@Article{MaraghBounitFadiliHammouri2014,
 Author = {Maragh, Fouad and Bounit, Hamid and Fadili, Ahmed and Hammouri, Hassan},
 Title = {On the admissible control operators for linear and bilinear systems and the {Favard} spaces},
 Journal = {Bulletin of the Belgian Mathematical Society - Simon Stevin},
 ISSN = {1370-1444},
 Volume = {21},
 Number = {4},
 Pages = {711--732},
 Year = {2014},
 Language = {English},
 zbMATH = {6374608},
 Zbl = {1301.93046},
doi={10.36045/bbms/1414091010}
}

@Book{HytoenenVanNeervenVeraarWeis2016,
 Author = {Hyt{\"o}nen, Tuomas and van Neerven, Jan and Veraar, Mark and Weis, Lutz},
 Title = {Analysis in {Banach} spaces. {Volume} {I}. {Martingales} and {Littlewood}-{Paley} theory},
 Series = {Ergebnisse der Mathematik und ihrer Grenzgebiete. 3. Folge},
 ISSN = {0071-1136},
 Volume = {63},
 ISBN = {978-3-319-48519-5; 978-3-319-48520-1},
 Year = {2016},
 Publisher = {Cham: Springer},
 Language = {English},
 DOI = {10.1007/978-3-319-48520-1},
 zbMATH = {6644896},
 Zbl = {1366.46001}
}

@Article{Monteiro2015,
 Author = {Monteiro, Giselle A.},
 Title = {On functions of bounded semivariation},
 Journal = {Real Analysis Exchange},
 ISSN = {0147-1937},
 Volume = {40},
 Number = {2},
 Pages = {233--276},
 Year = {2015},
 Language = {English},
 zbMATH = {6848835},
 Zbl = {1387.26021},
url={https://projecteuclid.org/journals/real-analysis-exchange/volume-40/issue-2/On-Functions-of-Bounded-Semivariation/rae/1491271216.full}
}

@Article{EberhardtGreiner1992,
 Author = {Eberhardt, Benjamin and Greiner, G{\"u}nther},
 Title = {Baillon's theorem on maximal regularity},
 Journal = {Acta Applicandae Mathematicae},
 ISSN = {0167-8019},
 Volume = {27},
 Number = {1-2},
 Pages = {47--54},
 Year = {1992},
 Language = {English},
 DOI = {10.1007/BF00046635},
 zbMATH = {417082},
 Zbl = {0802.47039}
}

@article{AroraGlueckSchwenningerPreprint,
      title={The lattice structure of negative Sobolev and extrapolation spaces}, 
      author={Sahiba Arora and Jochen Glück and Felix L. Schwenninger},
      year={2024},
      eprint={2404.02116v3},
      archivePrefix={arXiv},
      primaryClass={math.FA},
      note = {To appear in the Israel Journal of Mathematics}
}

@InCollection{AbramovichAliprantis2001,
 Author = {Abramovich, Yuriĭ A. and Aliprantis, Charalambos D.},
 Title = {Positive operators},
 BookTitle = {Handbook of the geometry of Banach spaces. Volume 1},
 ISBN = {0-444-82842-7},
 Pages = {85--122},
 Year = {2001},
 Publisher = {Amsterdam: Elsevier},
 Language = {English},
 zbMATH = {1684172},
 Zbl = {1202.47042},
 doi = {10.1016/S1874-5849(01)80004-8}
}

@Article{Ng1969,
 Author = {Ng, Kung-Fu},
 Title = {The duality of partially ordered {Banach} spaces},
 Journal = {Proceedings of the London Mathematical Society. Third Series},
 ISSN = {0024-6115},
 Volume = {19},
 Pages = {269--288},
 Year = {1969},
 Language = {English},
 DOI = {10.1112/plms/s3-19.2.269},
 zbMATH = {3271005},
 Zbl = {0169.15001}
}

@Book{Grafakos2014,
 Author = {Grafakos, Loukas},
 Title = {Classical {Fourier} analysis},
 Edition = {3rd ed.},
 Series = {Graduate Texts in Mathematics},
 ISSN = {0072-5285},
 Volume = {249},
 ISBN = {978-1-4939-1193-6; 978-1-4939-1194-3},
 Year = {2014},
 Publisher = {New York, NY: Springer},
 Language = {English},
 DOI = {10.1007/978-1-4939-1194-3},
 zbMATH = {6313565},
 Zbl = {1304.42001}
}

@article{ClementDiekmannGyllenbergHeijmansThieme1987,
	author = {Cl\'ement, Philippe J.E. and Diekmann, Odo and Gyllenberg, Mats and Heijmans, Henk J.A.M. and Thieme, Horst R.},
	doi = {10.1007/BF01457866},
	journal = {Math. Ann.},
	number = {4},
	pages = {709--725},
	title = {{Perturbation theory for dual semigroups. I.~The sun-reflexive case}},
	volume = {277},
	year = {1987},
	doi = {10.1007/BF01457866}
}

@article{ClementDiekmannGyllenbergHeijmansThieme1988,
	author = {Cl\'ement, Philippe J.E. and Diekmann, Odo and Gyllenberg, Mats and Heijmans, Henk J.A.M. and Thieme, Horst R.},
	doi = {10.1017/S0308210500026731},
	journal = {Proc. Roy. Soc. Edinburgh Sect. A},
	number = {1-2},
	pages = {145--172},
	title = {{Perturbation theory for dual semigroups. II.~Time-dependent perturbations in the sun-reflexive case}},
	volume = {109},
	year = {1988}
}

@article{ClementDiekmannGyllenbergHeijmansThieme1989a,
	author = {Cl\'ement, Philippe J.E. and Diekmann, Odo and Gyllenberg, Mats and Heijmans, Henk J.A.M. and Thieme, Horst R.},
     Title = {Perturbation theory for dual semigroups. {III}: {Nonlinear} {Lipschitz} continuous perturbations in the sunreflexive},
     Year = {1989},
     Language = {English},
     HowPublished = {Volterra integrodifferential equations in {Banach} spaces and applications, {Proc}. {Conf}., {Trento}/{Italy} 1987, {Pitman} {Res}. {Notes} {Math}. {Ser}},
    pages = {67--89},
volume={190},
     zbMATH = {4105672},
     Zbl = {0675.47036},
    url={https://core.ac.uk/download/pdf/301669584.pdf}
    }

@article{ClementDiekmannGyllenbergHeijmansThieme1989b,
	author = {Cl\'ement, Philippe J.E. and Diekmann, Odo and Gyllenberg, Mats and Heijmans, Henk J.A.M. and Thieme, Horst R.},
 Title = {Perturbation theory for dual semigroups. {IV}: {The} intertwining formula and the canonical pairing},
 Year = {1989},
 Language = {English},
 HowPublished = {Semigroup theory and applications, {Proc}. {Conf}., {Trieste}/{Italy} 1987, {Lect}. {Notes} {Pure} {Appl}. {Math}},
    volum={116},
 zbMATH = {4146910},
 Zbl = {0699.47028},
pages = {95--116},
url={https://ir.cwi.nl/pub/12515/12515D.pdf}
}

@InCollection{DiekmannGyllenbergThieme1991,
	author = {Diekmann, Odo and Gyllenberg, Mats and Thieme, Horst R.},
 Title = {Perturbation theory for dual semigroups. {V}: {Variation} of constants formulas},
 BookTitle = {Semigroup theory and evolution equations. The second international conference, held September 25 to 29, 1989, in Delft, Netherlands},
 ISBN = {978-0-8247-8545-1; 978-1-138-41748-9; 978-1-003-41991-4},
 Pages = {107--123},
 Year = {1991},
 Publisher = {New York etc.: Marcel Dekker, Inc.},
 Language = {English},
 zbMATH = {66223},
 Zbl = {0746.47019},
url={https://ir.cwi.nl/pub/1566/1566D.pdf}
}

@article{XuLiuYung2008,
 author = {Xu, Gen Qi and Liu, Chao and Yung, Siu Pang},
 title = {Necessary conditions for the exact observability of systems on {Hilbert} spaces},
 fjournal = {Systems \& Control Letters},
 journal = {Syst. Control Lett.},
 issn = {0167-6911},
 volume = {57},
 number = {3},
 pages = {222--227},
 year = {2008},
 language = {English},
 doi = {10.1016/j.sysconle.2007.08.006},
 zbMATH = {5237756},
 Zbl = {1130.93019}
}

@book{HillePhillips1974,
 author = {Hille, Einar and Phillips, Ralph S.},
 title = {Functional analysis and semi-groups. 3rd printing of rev. ed. of 1957},
 fseries = {Colloquium Publications. American Mathematical Society},
 series = {Colloq. Publ., Am. Math. Soc.},
 issn = {0065-9258},
 volume = {31},
 isbn = {0-8218-1031-6},
 year = {1974},
 publisher = {American Mathematical Society (AMS), Providence, RI},
 language = {English},
 url = {www.ams.org/online_bks/coll31/},
 zbMATH = {3608593},
 Zbl = {0392.46001}
}

@article{Carja1988,
 author = {C{\^a}rj{\u{a}}, Ovidiu},
 title = {On constraint controllability of linear systems in {Banach} spaces},
 fjournal = {Journal of Optimization Theory and Applications},
 journal = {J. Optim. Theory Appl.},
 issn = {0022-3239},
 volume = {56},
 number = {2},
 pages = {215--225},
 year = {1988},
 language = {English},
 doi = {10.1007/BF00939408},
 zbMATH = {4016686},
 Zbl = {0625.93011}
}

@article{Douglas1966,
 author = {Douglas, Ronald G.},
 title = {On majorization, factorization, and range inclusion of operators on {Hilbert} space},
 fjournal = {Proceedings of the American Mathematical Society},
 journal = {Proc. Am. Math. Soc.},
 issn = {0002-9939},
 volume = {17},
 pages = {413--415},
 year = {1966},
 language = {English},
 doi = {10.2307/2035178},
 zbMATH = {3235590},
 Zbl = {0146.12503}
}

@PhdThesis{Fröhlich2003,
  author = {Andreas M. {Fröhlich}},
  title  = {{ \(H^\infty\)-Kalkül und Dilatationen}},
  school = {Universität Karlsruhe},
  year   = {2003},
  language  = {German},
}

@article{FröhlichWeis2006,
 author = {Fr{\"o}hlich, Andreas M. and Weis, Lutz},
 title = {{{\(H^\infty\)}} calculus and dilations},
 fjournal = {Bulletin de la Soci{\'e}t{\'e} Math{\'e}matique de France},
 journal = {Bull. Soc. Math. Fr.},
 issn = {0037-9484},
 volume = {134},
 number = {4},
 pages = {487--508},
 year = {2006},
 language = {English},
 doi = {10.24033/bsmf.2520},
 zbMATH = {5229540},
 Zbl = {1168.47015}
}

@incollection{KunstmannWeis2004,
 author = {Kunstmann, Peer C. and Weis, Lutz},
 title = {Maximal {{\(L_p\)}}-regularity for parabolic equations, {Fourier} multiplier theorems and {{\(H^\infty\)}}-functional calculus},
 booktitle = {Functional analytic methods for evolution equations. Based on lectures given at the autumn school on evolution equations and semigroups, Levico Terme, Trento, Italy, October 28--November 2, 2001},
 isbn = {3-540-23030-0},
 pages = {65--311},
 year = {2004},
 publisher = {Berlin: Springer},
 language = {English},
 zbMATH = {2144723},
 Zbl = {1097.47041},
doi={10.1007/978-3-540-44653-8_2}
}

@book{KaltonLoristWeis2023,
 author = {Kalton, Nigel J. and Lorist, Emiel and Weis, Lutz},
 title = {Euclidean structures and operator theory in {Banach} spaces},
 fseries = {Memoirs of the American Mathematical Society},
 series = {Mem. Am. Math. Soc.},
 issn = {0065-9266},
 volume = {1433},
 isbn = {978-1-4704-6703-6; 978-1-4704-7575-8},
 year = {2023},
 publisher = {Providence, RI: American Mathematical Society (AMS)},
 language = {English},
 doi = {10.1090/memo/1433},
 zbMATH = {7726554},
 Zbl = {1533.47004}
}

@phdthesis{Kriegler2009,
    author       = {Kriegler, Christoph},
    year         = {2009},
    title        = {Spectral multipliers, R-bounded homomorphisms, and analytic diffusion semigroups},
    doi          = {10.5445/IR/1000015866},
    publisher    = {{Universität Karlsruhe (TH)}},
    language     = {english}
}

@article{HaakOuhabaz2012,
 author = {Haak, Bernhard H. and Ouhabaz, El Maati},
 title = {Exact observability, square functions and spectral theory},
 fjournal = {Journal of Functional Analysis},
 journal = {J. Funct. Anal.},
 issn = {0022-1236},
 volume = {262},
 number = {6},
 pages = {2903--2927},
 year = {2012},
 language = {English},
 doi = {10.1016/j.jfa.2012.01.007},
 zbMATH = {6017633},
 Zbl = {1241.93010}
}

\end{document}